\documentclass[10pt]{amsart}
\usepackage{amsmath}
\usepackage{amsthm}
\usepackage{amsfonts}
\usepackage{mathrsfs}
\usepackage{amssymb}
\usepackage{amscd}
\usepackage{pgf}
\usepackage{tikz}
\usetikzlibrary{cd}
\usetikzlibrary{decorations.pathreplacing, decorations.markings, calc}

\usepackage{amscd}

    \newtheorem{Lem}{Lemma}[section]
    \newtheorem{Lem-Def}{Lemma-Definition}[section]
    \newtheorem{Prop}[Lem]{Proposition}

    \newtheorem{Thm}[Lem]{Theorem}
    \newtheorem{Cor}[Lem]{Corollary}
		\newtheorem*{Thm*}{Theorem}

\theoremstyle{definition}

    \newtheorem{Def}[Lem]{Definition}
    \newtheorem{Exa}[Lem]{Example}
    \newtheorem{Rem}[Lem]{Remark}

		\newtheorem{Question}[Lem]{Question}

\expandafter\let\csname ver@amsthm.sty\endcsname\relax
\let\theoremstyle\relax

\usepackage[utf8]{inputenc}
\usepackage[
    breaklinks,
    colorlinks,
    citecolor=teal,
    linkcolor=teal,
    urlcolor=teal,
    pagebackref=true,
    hyperindex
    ]{hyperref}
\usepackage{fancyhdr}
\usepackage[
    hscale=0.7,
    vscale=0.75,
    headheight=13pt,
    centering,
    ]{geometry}
\usepackage{amsmath}
\usepackage{amsthm}
\usepackage{amssymb}
\usepackage{mathtools}
\usepackage{stmaryrd}
\usepackage{mathdots}
\usepackage[all,cmtip]{xy}
\usepackage{xcolor}
\usepackage{framed}
\usepackage{pgffor}
\usepackage{fnpct}
\usepackage[capitalize]{cleveref}

\theoremstyle{plain}

\theoremstyle{definition}

\numberwithin{figure}{section}

\numberwithin{equation}{section}

\makeatletter

\def\@myMR[#1 #2]{\relax\ifhmode\unskip\spacefactor3000 \space\fi
  \MRhref{#1}{MR\,#1}}
\renewcommand\MR[1]{\@myMR[#1 ]}
\renewcommand{\MRhref}[2]{{\tiny%
  \href{http://www.ams.org/mathscinet-getitem?mr=#1}{#2}}%
}
\renewcommand*{\backref}[1]{}
\renewcommand*{\backrefalt}[4]{%
    \tiny%
    ({
    \ifcase #1 not cited%
          \or cit.\ on p.~#2%
          \else cit.\ on pp.~#2%
    \fi%
    })\\[-.6em]}

\def\maketitle{\par
  \@topnum\z@ 
  \@setcopyright
  \thispagestyle{empty}
  \ifx\@empty\shortauthors \let\shortauthors\shorttitle
  \else \andify\shortauthors
  \fi
  \@maketitle@hook
  \begingroup
  \@maketitle
  \toks@\@xp{\shortauthors}\@temptokena\@xp{\shorttitle}%
  \toks4{\def\\{ \ignorespaces}}
  \edef\@tempa{%
    \@nx\markboth{\the\toks4
      \@nx\MakeUppercase{\the\toks@}}{\the\@temptokena}}%
  \@tempa
  \endgroup
  \c@footnote\z@
    \renewcommand{\footnoterule}{%
      \kern -3pt
      \hrule width \textwidth height .5pt
      \kern 2pt
    }
  {
    \renewcommand\thefootnote{}
    \vspace{-2em}
    \footnote{
      \par\vspace{-1.2em}\noindent
      \def\@footnotetext##1{\noindent{\footnotesize##1}\par}%
      \let\@makefnmark\relax  \let\@thefnmark\relax
      \ifx\@empty\@date\else \@footnotetext{\@setdate}\fi
      \ifx\@empty\@subjclass\else \@footnotetext{\@setsubjclass}\fi
      \ifx\@empty\@keywords\else \@footnotetext{\@setkeywords}\fi
      \ifx\@empty\thankses\else \@footnotetext{%
        \def\par{\let\par\@par}\@setthanks}%
      \fi
    }
    \addtocounter{footnote}{-1}
  }
  \@cleartopmattertags
}

\def\@adminfootnotes{\@empty}

\def\@settitle{\begin{center}%
  \baselineskip14\p@\relax
    \bfseries
\Large
  \@title
  \end{center}%
}

\def\@setauthors{%
  \begingroup
  \def\thanks{\protect\thanks@warning}%
  \trivlist
  \centering\footnotesize \@topsep30\p@\relax
  \advance\@topsep by -\baselineskip
  \item\relax
  \author@andify\authors
  \def\\{\protect\linebreak}%
  \large{\authors}%
  \ifx\@empty\contribs
  \else
    ,\penalty-3 \space \@setcontribs
    \@closetoccontribs
  \fi
  \endtrivlist
  \endgroup
}

\def\@setaddresses{\par
  \nobreak \begingroup
\footnotesize
  \def\author##1{\end{minipage}\hskip 1sp \begin{minipage}{.5\textwidth}\raggedright%
    ~\\[2em]{\bf##1}\\[.5em]%
  }%
  \interlinepenalty\@M
  \def\address##1##2{\begingroup
    {\ignorespaces##2}\endgroup\\[.5em]}%
  \def\curraddr##1##2{\begingroup
    \@ifnotempty{##2}{\nobreak\indent\curraddrname
      \@ifnotempty{##1}{, \ignorespaces##1\unskip}\/:\space
      ##2\par}\endgroup}%
  \def\email##1##2{\begingroup
    \@ifnotempty{##2}{\nobreak\indent
      \@ifnotempty{##1}{, \ignorespaces##1\unskip}
      \ttfamily##2\par}\endgroup}%
  \def\urladdr##1##2{\begingroup
    \def~{\char`\~}%
    \@ifnotempty{##2}{\nobreak\indent\urladdrname
      \@ifnotempty{##1}{, \ignorespaces##1\unskip}\/:\space
      \ttfamily##2\par}\endgroup}%
  \setlength{\parindent}{0pt}%
  \vfill%
  {
  \begin{minipage}{0mm}
  \addresses
  \end{minipage}
  }
  \endgroup
}

\renewcommand{\author}[2][]{%
  \ifx\@empty\authors
    \gdef\authors{#2}%
    \g@addto@macro\addresses{\author{#2}}%
  \else
    \g@addto@macro\authors{\and#2}%
    \g@addto@macro\addresses{\author{#2}}%
  \fi
  \@ifnotempty{#1}{%
    \ifx\@empty\shortauthors
      \gdef\shortauthors{#1}%
    \else
      \g@addto@macro\shortauthors{\and#1}%
    \fi
  }%
}
\edef\author{\@nx\@dblarg
  \@xp\@nx\csname\string\author\endcsname}

\def\@secnumfont{\@empty}

\def\section{\@startsection{section}{1}%
  \z@{.7\linespacing\@plus\linespacing}{.5\linespacing}%
  {\large\bfseries\centering}}

\makeatother

\usepackage{enumerate}

\usepackage{hyperref}

\newcommand{\A}{\mathcal A}

\newcommand{\col}{\colon}

\newcommand{\D}{\mathcal{D}}

\DeclareMathOperator{\roott}{root}

\DeclareMathOperator{\inv}{inv}
\DeclareMathOperator{\wt}{wt}
\DeclareMathOperator{\csf}{csf}
\DeclareMathOperator{\LLT}{LLT}
\DeclareMathOperator{\asc}{asc}
\DeclareMathOperator{\lrv}{lrv}

\newcommand{\m}{\mathbf{m}}

\newcommand{\F}{\mathcal{F}}

\title{A symmetric function of increasing forests}
\author{Alex Abreu}

\address{
    Instituto de Matemática e Estatística\\
    Universidade Federal Fluminense\\
    Rua Prof. M. W. de Freitas, S/N\\
    24210-201 Niterói, Rio de Janeiro, Brasil
}

\email{alexbra1@gmail.com}

\author{Antonio Nigro}

\address{
    Instituto de Matemática e Estatística\\
    Universidade Federal Fluminense\\
    Rua Prof. M. W. de Freitas, S/N\\
    24210-201 Niterói, Rio de Janeiro, Brasil
}

\email{antonio.nigro@gmail.com}

\newcommand{\dyckpath}[2]{
\draw[line width=2pt] (#1) foreach \dir in {#2}{ -- ++(\dir*90:1)};
}

\newcommand{\g}[3]{
\begin{tikzpicture}
\begin{scope}[decoration={markings,mark=at position 0.6 with {\arrow[scale=0.5]{>}}}, scale=0.3, thick ]
\ifnum#1=-1 \draw[postaction={decorate}] (0,1) -- (0,0); \fi
\ifnum#1=1 \draw[postaction={decorate}] (0,0) -- (0,1); \fi
\ifnum#1=0 \draw (0,1) -- (0,0); \fi
\ifnum#2=1 \draw[postaction={decorate}] (0,0) -- (1,0); \fi
\ifnum#2=-1 \draw[postaction={decorate}] (1,0) -- (0,0); \fi
\ifnum#2=0 \draw (1,0) -- (0,0); \fi
\ifnum#3=1 \draw[postaction={decorate}] (1,0) -- (0,1); \fi
\ifnum#3=-1 \draw[postaction={decorate}] (0,1) -- (1,0); \fi
\ifnum#3=0 \draw (1,0) -- (0,1); \fi
\draw (0,1) circle (0.1cm) [fill=black];
\draw (0,0) circle (0.1cm) [fill=black];
\draw (1,0) circle (0.1cm) [fill=black];
\end{scope}
\end{tikzpicture}
}

\newcommand{\G}[3]{
\begin{tikzpicture}
\begin{scope}[decoration={markings,mark=at position 0.6 with {\arrow[scale=0.7]{>}}}, scale=0.6, thick ]
\ifnum#1=-1 \draw[postaction={decorate}] (0,1) -- (0,0); \fi
\ifnum#1=1 \draw[postaction={decorate}] (0,0) -- (0,1); \fi
\ifnum#1=0 \draw (0,1) -- (0,0); \fi
\ifnum#2=1 \draw[postaction={decorate}] (0,0) -- (1,0); \fi
\ifnum#2=-1 \draw[postaction={decorate}] (1,0) -- (0,0); \fi
\ifnum#2=0 \draw (1,0) -- (0,0); \fi
\ifnum#3=1 \draw[postaction={decorate}] (1,0) -- (0,1); \fi
\ifnum#3=-1 \draw[postaction={decorate}] (0,1) -- (1,0); \fi
\ifnum#3=0 \draw (1,0) -- (0,1); \fi
\draw (0,1) circle (0.1cm) [fill=black];
\draw (0,0) circle (0.1cm) [fill=black];
\draw (1,0) circle (0.1cm) [fill=black];
\end{scope}
\end{tikzpicture}
}

\newcommand{\ga}[3]{
\begin{tikzpicture}
\begin{scope}[decoration={markings,mark=at position 0.6 with {\arrow[scale=0.5]{>}}}, scale=0.3, thick ]
\ifnum#3=-1 \draw[postaction={decorate}] (0,1) -- (0,0); \fi
\ifnum#3=1 \draw[postaction={decorate}] (0,0) -- (0,1); \fi
\ifnum#3=0 \draw (0,1) -- (0,0); \fi
\ifnum#2=1 \draw[postaction={decorate}] (0,1) -- (1,1); \fi
\ifnum#2=-1 \draw[postaction={decorate}] (1,1) -- (0,1); \fi
\ifnum#2=0 \draw (1,1) -- (0,1); \fi
\ifnum#1=1 \draw[postaction={decorate}] (0,0) -- (1,1); \fi
\ifnum#1=-1 \draw[postaction={decorate}] (1,1) -- (0,0); \fi
\ifnum#1=0 \draw (1,1) -- (0,0); \fi
\draw (0,1) circle (0.1cm) [fill=black];
\draw (0,0) circle (0.1cm) [fill=black];
\draw (1,1) circle (0.1cm) [fill=black];
\end{scope}
\end{tikzpicture}
}

\newcommand{\Ga}[3]{
\begin{tikzpicture}
\begin{scope}[decoration={markings,mark=at position 0.6 with {\arrow[scale=0.7]{>}}}, scale=0.6, thick ]
\ifnum#3=-1 \draw[postaction={decorate}] (0,1) -- (0,0); \fi
\ifnum#3=1 \draw[postaction={decorate}] (0,0) -- (0,1); \fi
\ifnum#3=0 \draw (0,1) -- (0,0); \fi
\ifnum#2=1 \draw[postaction={decorate}] (0,1) -- (1,1); \fi
\ifnum#2=-1 \draw[postaction={decorate}] (1,1) -- (0,1); \fi
\ifnum#2=0 \draw (1,1) -- (0,1); \fi
\ifnum#1=1 \draw[postaction={decorate}] (0,0) -- (1,1); \fi
\ifnum#1=-1 \draw[postaction={decorate}] (1,1) -- (0,0); \fi
\ifnum#1=0 \draw (1,1) -- (0,0); \fi\draw (0,1) circle (0.1cm) [fill=black];
\draw (0,0) circle (0.1cm) [fill=black];
\draw (1,0) circle (0.1cm) [fill=black];
\draw (1,1) circle (0.1cm) [fill=black];
\end{scope}
\end{tikzpicture}
}

\begin{document}
\maketitle

\begin{abstract}
   For an indifference graph $G$ we define a symmetric function of increasing spanning forests of $G$. We prove that this symmetric function satisfies certain linear relations, which are also satisfied by the chromatic quasisymmetric function and unicellular $\LLT$ polynomials. As a consequence we give a combinatorial interpretation of the coefficients of the $\LLT$ polynomial in the elementary basis (up to a factor of a power of $(q-1)$), strengthening the description given in \cite{AS2020}.
\end{abstract}

\tableofcontents
\section{Introduction}

      In 1730 Stirling introduced what are now called (signless) Stirling numbers of the first kind. These numbers $s(n,k)$ may be defined via the following equality
      \begin{equation}
      \label{eq:stir}
      x(x-1)\ldots (x-n+1)=\sum (-1)^{n-k}s(n,k)x^k.
      \end{equation}
      One way to interpret this elementary identity is to notice that the left-hand side is the chromatic polynomial of the complete graph $K_n$ on $n$ vertices, so the numbers $s(n,k)$ are the coefficients of the chromatic polynomial in the basis $x^k$. A combinatorial interpretation for the $s(n,k)$ is the number of increasing forests with $k$ components and vertex set $[n]=\{1,\ldots,n\}$ (see \cite{Bergerontree}).  An \emph{increasing tree} $T$ is a tree with totally ordered vertex set such that the vertices along any path starting on the minimal vertex go in increasing order (see Figure \ref{fig:inctree}).  An \emph{increasing forest} is a forest with totally ordered vertex set such that its components are increasing trees.
      \begin{figure}[htb]
          \centering
          \begin{tikzpicture}
          \begin{scope}
          \node (1) at (0,0) [label=below:$1$, shape=circle, fill=black, inner sep=2pt] {};
           \node (2) at (1,0.5) [label=below:$2$, shape=circle, fill=black,inner sep=2pt] {};
           \node (3) at (1,-0.5) [label=below:$3$, shape=circle, fill=black,inner sep=2pt] {};
           \node (4) at (2,0.5) [label=below:$4$, shape=circle, fill=black,inner sep=2pt] {};
           \node (5) at (3,0) [label=below:$5$, shape=circle, fill=black,inner sep=2pt] {};
           \node (6) at (3,1) [label=below:$6$, shape=circle, fill=black,inner sep=2pt] {};
           \node (7) at (2,-0.5) [label=below:$7$, shape=circle, fill=black,inner sep=2pt] {};
           \draw (1) -- (2) -- (4)-- (5);
           \draw (4) -- (6);
           \draw (1) -- (3) -- (7);
           \end{scope}
           \begin{scope}[shift={(5,0)}]
           \node (1) at (0,0) [label=below:$1$, shape=circle, fill=black, inner sep=2pt] {};
           \node (2) at (1,0.5) [label=below:$2$, shape=circle, fill=black,inner sep=2pt] {};
           \node (3) at (1,-0.5) [label=below:$3$, shape=circle, fill=black,inner sep=2pt] {};
           \node (4) at (2,0.5) [label=below:$7$, shape=circle, fill=black,inner sep=2pt] {};
           \node (5) at (3,0) [label=below:$5$, shape=circle, fill=black,inner sep=2pt] {};
           \node (6) at (3,1) [label=below:$6$, shape=circle, fill=black,inner sep=2pt] {};
           \node (7) at (2,-0.5) [label=below:$4$, shape=circle, fill=black,inner sep=2pt] {};
           \draw[color=red] (1) -- (2) -- (4)-- (5);
           \draw (4) -- (6);
           \draw (1) -- (3) -- (7);
           \end{scope}
          \end{tikzpicture}
          \caption{An increasing tree (left) and a non-increasing tree (right) with a non-increasing path highlighted.}
          \label{fig:inctree}
      \end{figure}
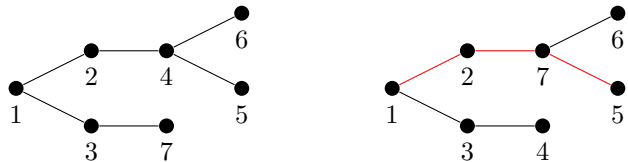

        With these definitions, we see that the coefficient of $x^k$ in the chromatic polynomial of $K_n$ counts (up to a sign) the number of increasing spanning forests of $K_n$ with $k$ components. \par

       This interpretation actually holds for a larger class of graphs, namely, graphs that have a perfect elimination ordering. A \emph{perfect elimination ordering} is an ordering $v_1,v_2,\ldots, v_n$ of the vertex set of $G$ such that, for each vertex $v_j$, the set $\{v_i; i<j, \{v_i,v_j\}\in E(G)\}$ induces a clique of $G$. In this article, we will assume that $G$ has vertex set $[n]$ and that the usual ordering of $[n]$ is a perfect elimination ordering. Then  \cite[Theorem 26]{HallamSagan} states that the coefficient of $x^k$ in $\chi_G(x)$ counts (up to a sign $(-1)^k$) the number of increasing spanning forests of $G$ with exactly $k$ components.\par
      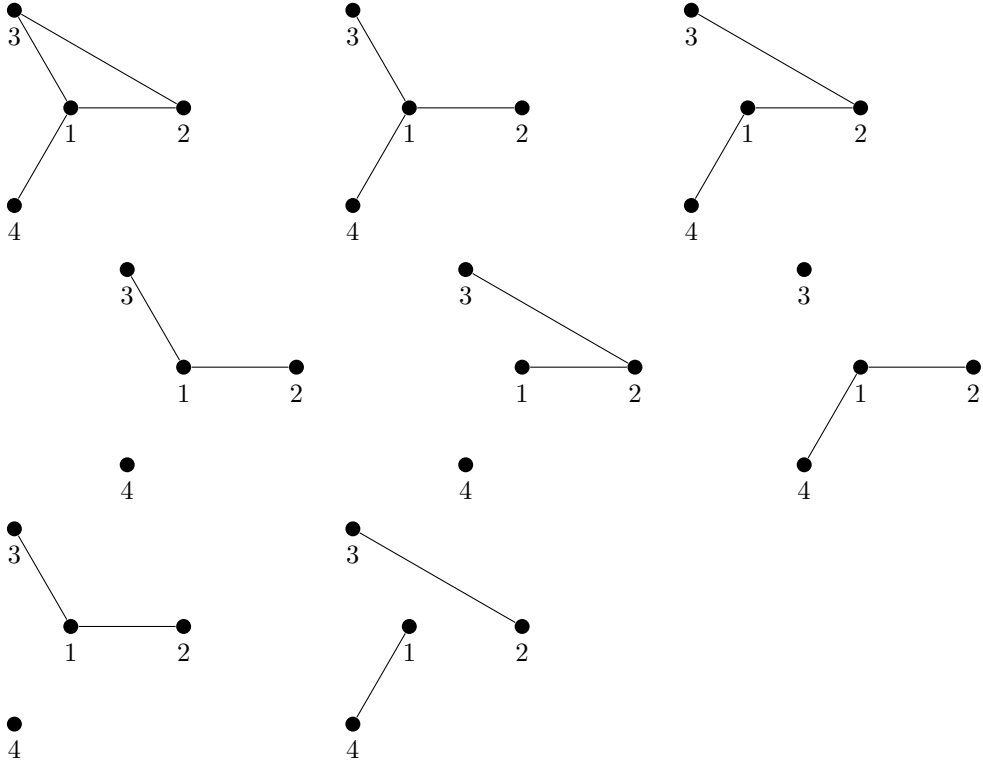
\begin{figure}[htb]
          \centering
          \begin{tikzpicture}[scale=1.5]
         \begin{scope}
           \node (1) at (0,0) [label=below:$1$, shape=circle, fill=black, inner sep=2pt] {};
           \node (2) at (1,0) [label=below:$2$, shape=circle, fill=black,inner sep=2pt] {};
           \node (3) at (-0.5,0.866) [label=below:$3$, shape=circle, fill=black,inner sep=2pt] {};
           \node (4) at (-0.5,-0.866) [label=below:$4$, shape=circle, fill=black,inner sep=2pt] {};
           \draw (1) -- (2);
           \draw (1) -- (3);
           \draw (1) -- (4);
           \draw (2) -- (3);
          \end{scope}
          \begin{scope}[shift={(3,0)}]
           \node (1) at (0,0) [label=below:$1$, shape=circle, fill=black, inner sep=2pt] {};
           \node (2) at (1,0) [label=below:$2$, shape=circle, fill=black,inner sep=2pt] {};
           \node (3) at (-0.5,0.866) [label=below:$3$, shape=circle, fill=black,inner sep=2pt] {};
           \node (4) at (-0.5,-0.866) [label=below:$4$, shape=circle, fill=black,inner sep=2pt] {};
           \draw (1) -- (2);
           \draw (1) -- (3);
           \draw (1) -- (4);
          \end{scope}
          \begin{scope}[shift={(6,0)}]
           \node (1) at (0,0) [label=below:$1$, shape=circle, fill=black, inner sep=2pt] {};
           \node (2) at (1,0) [label=below:$2$, shape=circle, fill=black,inner sep=2pt] {};
           \node (3) at (-0.5,0.866) [label=below:$3$, shape=circle, fill=black,inner sep=2pt] {};
           \node (4) at (-0.5,-0.866) [label=below:$4$, shape=circle, fill=black,inner sep=2pt] {};
           \draw (1) -- (2);
           \draw (2) -- (3);
           \draw (1) -- (4);
          \end{scope}
          \begin{scope}[shift={(1,-2.3)}]
           \node (1) at (0,0) [label=below:$1$, shape=circle, fill=black, inner sep=2pt] {};
           \node (2) at (1,0) [label=below:$2$, shape=circle, fill=black,inner sep=2pt] {};
           \node (3) at (-0.5,0.866) [label=below:$3$, shape=circle, fill=black,inner sep=2pt] {};
           \node (4) at (-0.5,-0.866) [label=below:$4$, shape=circle, fill=black,inner sep=2pt] {};
           \draw (1) -- (2);
           \draw (1) -- (3);
            \end{scope}
          \begin{scope}[shift={(4,-2.3)}]
           \node (1) at (0,0) [label=below:$1$, shape=circle, fill=black, inner sep=2pt] {};
           \node (2) at (1,0) [label=below:$2$, shape=circle, fill=black,inner sep=2pt] {};
           \node (3) at (-0.5,0.866) [label=below:$3$, shape=circle, fill=black,inner sep=2pt] {};
           \node (4) at (-0.5,-0.866) [label=below:$4$, shape=circle, fill=black,inner sep=2pt] {};
           \draw (1) -- (2);
           \draw (2) -- (3);
           \end{scope}
          \begin{scope}[shift={(7,-2.3)}]
           \node (1) at (0,0) [label=below:$1$, shape=circle, fill=black, inner sep=2pt] {};
           \node (2) at (1,0) [label=below:$2$, shape=circle, fill=black,inner sep=2pt] {};
           \node (3) at (-0.5,0.866) [label=below:$3$, shape=circle, fill=black,inner sep=2pt] {};
           \node (4) at (-0.5,-0.866) [label=below:$4$, shape=circle, fill=black,inner sep=2pt] {};
           \draw (1) -- (2);
           \draw (1) -- (4);
          \end{scope}
          \begin{scope}[shift={(0,-4.6)}]
           \node (1) at (0,0) [label=below:$1$, shape=circle, fill=black, inner sep=2pt] {};
           \node (2) at (1,0) [label=below:$2$, shape=circle, fill=black,inner sep=2pt] {};
           \node (3) at (-0.5,0.866) [label=below:$3$, shape=circle, fill=black,inner sep=2pt] {};
           \node (4) at (-0.5,-0.866) [label=below:$4$, shape=circle, fill=black,inner sep=2pt] {};
           \draw (1) -- (2);
           \draw (1) -- (3);
            \end{scope}
          \begin{scope}[shift={(3,-4.6)}]
           \node (1) at (0,0) [label=below:$1$, shape=circle, fill=black, inner sep=2pt] {};
           \node (2) at (1,0) [label=below:$2$, shape=circle, fill=black,inner sep=2pt] {};
           \node (3) at (-0.5,0.866) [label=below:$3$, shape=circle, fill=black,inner sep=2pt] {};
           \node (4) at (-0.5,-0.866) [label=below:$4$, shape=circle, fill=black,inner sep=2pt] {};
           \draw (2) -- (3);
           \draw (1) -- (4);
           \end{scope}
          \end{tikzpicture}
          \caption{A graph $G$ with a perfect elimination ordering and chromatic polynomial $x^4-4x^3+5x^2-2x$, its $2$ increasing spanning forests with $1$ component and its $5$ increasing spanning forests with $2$ components.}
          \label{fig:exampletree}
      \end{figure}

     The chromatic polynomial of a graph admits a symmetric function generalization introduced by Stanley in \cite{Stan95}. Given a graph $G$, it is defined as
     \[
     \csf(G):=\sum_{\kappa}x_{\kappa},
     \]
     where the sum runs through all proper colorings of the vertices  $\kappa\col V(G)\to \mathbb{N}$ and $x_{\kappa}:=\prod_{v\in V(G)} x_{\kappa(v)}$. A coloring $\kappa$ is proper if $\kappa(v)\neq\kappa(v')$ whenever $v$ and $v'$ are adjacent. We have that $\csf(G)$ lies in $\Lambda$, the algebra of symmetric functions.\par
     The algebra $\Lambda$ has three important sets of generators, called the elementary, power sum, and complete symmetric functions:
     \begin{align*}
         e_n=\sum_{i_1<i_2<\ldots <i_n}x_{i_1}x_{i_2}\ldots x_{i_n},\quad\quad         p_n=\sum_i x_i^n,\quad\quad         h_n=\sum_{i_1\leq i_2\leq\ldots \leq i_n}x_{i_1}x_{i_2}\ldots x_{i_n}.
    \end{align*}
    Every element in $\Lambda$ can be written as a  polynomial in one of the generators above. Moreover, for every partition $\lambda=(\lambda_1,\ldots, \lambda_{\ell(\lambda)})$ we define $e_{\lambda}:=\prod e_{\lambda_i}$, $p_{\lambda}:=\prod p_{\lambda_i}$ and $h_{\lambda}:=\prod h_{\lambda_i}$. Each of the sets $\{e_{\lambda}\}$, $\{p_{\lambda}\}$ and $\{h_{\lambda}\}$ is a homogeneous basis for $\Lambda$.\par

          If we set $x_1=x_2=\ldots=x_m=1$ and $x_{m+1}=x_{m+2}=\ldots=0$, then $\csf(G)(1,\ldots, 1,0,\ldots)$ is precisely the number of ways to color $G$ with $m$ colors, which is simply $\chi_G(n)$. Moreover, via this substitution we have that $e_n={m\choose n}$, $h_n={m+n-1\choose n}$ and $p_n=m$. We can actually make $m$ into an indeterminate $x$. If we define the \emph{specialization map} $\epsilon\col\Lambda\to \mathbb{Q}[x]$ as the homomorphism of algebras given by $\epsilon(p_n)=x$ for every $n$, then we have that $\epsilon(e_n)=x(x-1)\ldots(x-n+1)/n!$, $\epsilon(h_n)=x(x+1)\ldots(x+n-1)/n!$, and, as expected, $\epsilon(\csf(G))=\chi_G$. We refer to \cite[Example 1, page 26]{Macdonald} for more details about the specialization map.\par
       Via the specialization map, we have $\epsilon(p_{\lambda})=x^{\ell(\lambda)}$ for every partition $\lambda$. So, the analogue of writing $\chi_G$ in the basis $x^k$ is to express $\csf(G)$ in the basis $p_{\lambda}$ with $\lambda\vdash n$. For instance, if $G$ is the graph in Figure \eqref{fig:exampletree}, we have that
       \[
       \csf(G)=p_{1, 1, 1, 1} - 4p_{2, 1, 1} + p_{2, 2} + 4p_{3, 1} - 2p_{4}.
       \]
       The attentive reader may notice that the five increasing spanning forests of $G$ with two components are nicely divided into two groups: There are four forests with components of either $3$ vertices or $1$ vertex, and one forest with components having $2$ vertices. With that in mind, we define the partition $\lambda(F)$ of a forest $F$ as the partition induced by the number of vertices of the components of $F$.\par
       Going back to the complete graph, the chromatic symmetric function is $n!e_n$. In that case, we have the well-known Newton identity
      \begin{equation}
      \label{eq:enp}
         n!e_n=\sum_{\lambda\vdash n} (-1)^{n-\ell(\lambda)}a_\lambda p_\lambda
      \end{equation}
      where $a_\lambda$ is the number of permutations in $n$ elements with cycle partition $\lambda$. The number $a_{\lambda}$ also counts increasing forests with vertex set $[n]$ and partition $\lambda$.  Indeed, we can construct a bijection between permutations with size $n$ and increasing forests with vertex set $[n]$ as follows (see \cite[Example 1.3.15]{Stanleyenu1}). For each cycle $\tau=(j_1,j_2,\ldots, j_k)$, with $j_1\leq j_l$ for every $l\in[k]$, we construct an increasing tree $T_{\tau}$ with vertex set $\{j_1,\ldots, j_k\}$ where $j_i$ is connected to the rightmost element $j_l$ of $\tau$ which precedes $j_i$ and which is less than $j_i$. If $\sigma$ is a permutation in $S_n$, write $\sigma$ in cyclic notation $\sigma=\tau_1\ldots \tau_k$ and define the forest $F_{\sigma}$ as $T_{\tau_1}\cup T_{\tau_2}\cup\ldots\cup T_{\tau_k}$. It follows from the construction that if $\sigma$ has cycle partition $\lambda$ then $\lambda(F_{\sigma})=\lambda$ as well. See Figure \eqref{fig:permtree} for an example.\par
       
       \begin{figure}[htb]
           \begin{tikzpicture}
              \begin{scope}
        \node (1) at (1,0) [label=below:$1$, shape=circle, fill=black, inner sep=2pt] {};
        \node (2) at (2,0) [label=below:$2$, shape=circle, fill=black,inner sep=2pt] {};
        \node (3) at (3,0) [label=below:$3$, shape=circle, fill=black,inner sep=2pt] {};
        \node (4) at (4,0) [label=below:$4$, shape=circle, fill=black,inner sep=2pt] {};
        \node (5) at (5,0) [label=below:$5$, shape=circle, fill=black,inner sep=2pt] {};
        \node (6) at (6,0) [label=below:$6$, shape=circle, fill=black,inner sep=2pt] {};
        \node (7) at (7,0) [label=below:$7$, shape=circle, fill=black,inner sep=2pt] {};
        \node (8) at (8,0) [label=below:$8$, shape=circle, fill=black, inner sep=2pt] {};
        \node (9) at (9,0) [label=below:$9$, shape=circle, fill=black, inner sep=2pt] {};
        \draw (1) to [bend left=45] (6);
        \draw (2) to [bend left=45] (4);
        \draw (1) to (2);
        \draw[color=red] (3) to [bend left=45] (7);
        \draw[color=red] (3) to [bend left=45] (5);
        \draw[color=red] (7) to [bend left=45] (9);
        \draw (4) to [bend left=45] (8);
        \end{scope}
          \end{tikzpicture}
           \caption{The increasing forest associated to the permutation $\sigma=(16248)(3795)$}
           \label{fig:permtree}
       \end{figure}
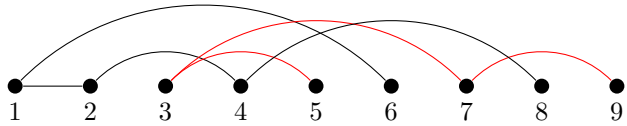

       Actually, if $G$ is a graph with a perfect elimination ordering,
       \begin{equation}
       \label{eq:csfGtree}
       \csf(G)=\sum_{F}(-1)^{n-\ell(\lambda(F))}p_{\lambda(F)}
       \end{equation}
       where the sum runs through all increasing spanning forests of $G$.\footnote{The authors could not find this precise statement in the literature, but it readily follows from \cite[Corollary 2.7]{Stan95} and \cite{HallamSagan}.}\par

      The Stirling numbers of the first kind admit a $q$-deformation $s_q(n,k)$ introduced by Gould \cite{Gould}, defined through the identity
      \[
      x(x-[1]_q)(x-[2]_q)\ldots (x-[n-1]_q)=\sum (-1)^{n-k}s_q(n,k)x^k,
      \]
      where $[j]_q=(q^j-1)/(q-1)$. Since $s(n,k)$ enumerates increasing forests, we could ask what statistic on increasing forests gives $s_q(n,k)$. In another direction, we could ask  what is a $q$-analogue of Equation \eqref{eq:enp}. Of course, we would have first to modify either $e_n$ or $p_\lambda$.\par

      To simplify matters and avoid unnecessary signs, we will apply the usual involution $\omega$ to Equation \eqref{eq:enp}. Recall that $\omega\col \Lambda\to \Lambda$ is defined by $\omega(e_n):=h_n$ and satisfies $\omega(p_n)=(-1)^{n-1}p_n$. Then equation \eqref{eq:enp} becomes an additional Newton identity:
      \begin{equation}
      \label{eq:nhp}
      n!h_n=\sum a_{\lambda}p_{\lambda}.
      \end{equation}
      Another relation involving the power sum and complete symmetric functions is the recursion
      \[
      nh_n=\sum_{i=1}^n h_{n-i}p_i.
      \]
      With this in mind, we define the symmetric functions $\rho_n$ in $\Lambda_q:=\Lambda[q]$ by the recursion
      \[
      [n]_qh_n=\sum_{i=1}^n h_{n-i}\rho_i.
      \]
      For instance, we have that $\rho_1=h_1$, $\rho_2=(q+1)h_2-h_{1,1}$, and $\rho_3=(q^2+q+1)h_3-(q+2)h_{2,1}+h_{1,1,1}$. We note that  $\rho_n$ interpolates between the power sum and the elementary symmetric function, in the sense that $\rho_n(1)=p_n$ and $\rho_n(0)=(-1)^{n-1}e_n$. The reader familiar with symmetric functions will quickly realize that $\rho_n$ is actually a modification of the Hall-Littlewood polynomial $P_n$, namely  $\rho_n=q^{n-1}P_n(q^{-1})$. Defining $\rho_{\lambda}:=\prod \rho_{\lambda_i}$ we get a $q$-analogue of Equation \eqref{eq:nhp},
      \[
      n!_qh_n=\sum a_\lambda(q)\rho_\lambda,
      \]
      where $n!_q:=\prod_{j=1}^n[j]_q$ and $a_{\lambda}(q)$ is a polynomial with non-negative integer coefficients. Moreover, we have that $s_{q}(n,k)=\sum a_{\lambda}(q)$, where the sum runs through all partitions $\lambda$ of $n$ with length $\ell(\lambda)=k$ (see Proposition \ref{prop:rhoeq} and Corollary \ref{cor:stir}). \par


      Since $s(n,k)$ and $a_{\lambda}$ enumerate increasing forests, it is natural to ask if there is a statistic on increasing forests that gives $s_q(n,k)$ and $a_{\lambda}(q)$, and more generally if there is a $q$-analogue of Equation \eqref{eq:csfGtree}.

      One of the main goals of this article is to answer this question. We first notice that the left-hand side of Equation \eqref{eq:csfGtree} already has a $q$-analogue introduced in \cite{ShareshianWachs}, called the \emph{chromatic quasisymmetric function} of $G$.  For a graph $G$ with set of vertices $[n]$, the chromatic quasisymmetric function $\csf_q(G)$ is defined as
    \[
    \csf_q(G):=\sum_{\kappa}q^{\asc_G(k)} x_\kappa.
    \]
    where the sum runs through all proper colorings of $G$  and
    \[
    \asc_G(\kappa):=|\{(i,j);i<j, \kappa(i)<\kappa(j); \{i,j\}\in E(G)\}|
    \]
    is the number of ascents of the coloring $\kappa$.\par
     For the right-hand side of Equation \eqref{eq:csfGtree}, $\rho_\lambda$ will play the part of $p_\lambda$, so it remains to find a statistic $\wt_G(F)$ for every increasing spanning forest $F$ of $G$ such that
     \[
     \csf_q(G)=\sum_{F}q^{\wt_G(F)}\rho_{\lambda(F)}.
     \]
     One complication is that $\csf_q(G)$ is not usually symmetric, in which case there is no chance for the equality above to hold. To avoid this, we will restrict our attention in the rest of this article to a class of graphs for which $\csf_q(G)$ is known to be symmetric, namely indifference graphs.\par

    A graph is called an \emph{indifference graph} if its vertex set is $[n]:=\{1,2,\ldots, n\}$ for some natural $n$ and such that if $\{i,j\}$ is an edge with $i<j$, then $\{i,k\}$ and $\{k,j\}$ are also edges for every $k$ such that $i<k<j$. In particular, on every indifference graph the usual ordering of $[n]$ is a perfect elimination ordering.  \par

     Indifference graphs are naturally associated with Hessenberg functions. A \emph{Hessenberg function} is a non-decreasing function $\m\col [n]\to[n]$ such that $\m(i)\geq i$ for every $i\in [n]$. The graph $G_{\m}$ induced by $\m$ is the graph with vertex set $[n]$ and edge set $\{(i,j); i<j\leq \m(i)\}$. Every indifference graph arises in this way.\par

     \begin{figure}[htb]
       \begin{tikzpicture}
       \begin{scope}[scale=0.6]
       \draw[help lines] (0,0) grid +(3,3);
       \draw[fill=blue, fill opacity=0.2, draw opacity=0] (0,1) rectangle +(1,1);
       \draw[fill=red, fill opacity=0.2, draw opacity=0] (1,2) rectangle +(1,1);
       \dyckpath{0,0}{1,1,0,1,0,0}
       \node at (1.5,-1) {$h=(2,3,3)$};
       \end{scope}
       \begin{scope}[scale=0.6, shift={(4,0)}]
        \node (1) at (1,1.5) [label=below:$1$, shape=circle, fill=black, inner sep=2pt] {};
        \node (2) at (3,1.5) [label=below:$2$, shape=circle, fill=black,inner sep=2pt] {};
        \node (3) at (5,1.5) [label=below:$3$, shape=circle, fill=black,inner sep=2pt] {};
        \node at (3,-1) {$G$};
        \draw[color=blue] (1) to (2);
        \draw[color=red] (2) to (3);
        \end{scope}
        \end{tikzpicture}
        \caption{A Hessenberg function $h$ and its associated indifference graph.}
       \label{fig:fig1}
   \end{figure}
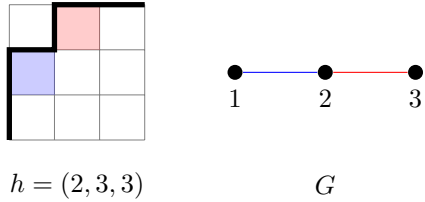

      For a permutation $\sigma\in S_n$ we say that $\sigma\leq \m$ if $\sigma(i)\leq \m(i)$ for every $i\in [n]$ and we denote by $S_{n,\m}$ the set of all permutations $\sigma\in S_n$ such that $\sigma\leq \m$. 
     One easy observation is that $\sigma\leq\m$ if and only if $F_{\sigma}$ is a increasing spanning forest of $G_{\m}$. Equation \eqref{eq:csfGtree} then becomes
     \[
       \omega(\csf(G))=\sum_{\sigma\leq \m} p_{\lambda(\sigma)}.
       \]
       This was known \emph{ante litteram} in the work of Stanley and Stembridge \cite{StanStem}. We mention this as it is easier to define the statistic $\wt_G(F)$ in terms of the associated permutation.\par

       A $G$-inversion of $\sigma$ is a pair $(i,j)$ with $i<j$, $\sigma(i)>\sigma(j)$, and $\{\sigma(j),\sigma(i)\}\in E(G)$. Also, we let $\sigma^c$ be the permutation obtained by removing the parenthesis in the cyclic notation of $\sigma$. Then we define $\wt_G(\sigma)$ to be the number of $G$-inversions of $\sigma^c$. We have the following theorem
      \begin{Thm}
      \label{thm:main1}
      If $G$ is an indifference graph, then
      \begin{equation}
          \label{eq:thm1}
          \omega(\csf_q(G))=\sum_{\sigma\leq \m}q^{\wt_G(\sigma)}\rho_{\lambda(\sigma)}=\sum_{\lambda\vdash n}c_{\lambda}(G)\rho_{\lambda(\sigma)}.
      \end{equation}
      \end{Thm}

      Increasing forests also play a role in another symmetric function related to colorings of $G$. If $G$ is an indifference graph, the unicellular $\LLT$ polynomial associated to $G$ is defined as
     \[
      \LLT(G):=\sum_{\kappa}q^{\asc_G(\kappa)}x_{\kappa}
    \]
    where the sum runs through all (not necessarily proper) colorings of $G$. \footnote{These polynomials can actually be defined in a more general setting, see \cite{LLT}}. \par
     Since $\csf_q$ and the unicellular $\LLT$ polynomial are closely related by a plethystic transformation (see \cite{CarlssonMellit}),  we also get

      \begin{Thm}
      \label{thm:main2}
      If $G$ is an indifference graph, then
      \begin{equation}
          \label{eq:thm2}
          \LLT(G)=\sum_{\sigma\leq \m}(q-1)^{n-\ell(\lambda(\sigma))}q^{\wt_G(\sigma)}e_{\lambda(\sigma)}.
      \end{equation}
            In particular, we have that
      \begin{equation}
          \label{eq:LLTN}
          \LLT(G)=\sum_{\lambda\vdash n}(q-1)^{n-\ell(\lambda)}c_\lambda(G)e_{\lambda},
      \end{equation}
      where $c_{\lambda}(G)\in \mathbb{N}[q]$.
      \end{Thm}

      As a direct corollary we obtain

      \begin{Cor}
      \label{cor:main}
      When $G$ is an indifference graph, we have that $\LLT(G;q+1)$ is $e$-positive.
      \end{Cor}

    The $e$-positivity of the vertical strip $\LLT$ polynomials, after the shift $q\to q+1$, was first conjectured in \cite{Alexandersson_2020} and \cite{garsia2019epositivity}, following \cite{Bergeron} where several symmetric functions were conjectured to be $e$-positive after this shift. This conjecture has been recently proved in \cite[Corollary 5.7]{dadderio} and \cite[Corollary 2.10]{AS2020}. The latter article also proves the explicit combinatorial description conjectured in \cite{Alexandersson_2020}. This description also follows from Theorem \ref{thm:main2} (see Corollary \ref{cor:peralex}).\par

      We note that Theorem \ref{thm:main2} is a slight strengthening (in the unicellular case) of this conjecture, in the sense that we isolate the explicit contribution of the $(q-1)$ factors. Moreover, with a more careful examination of Theorem \ref{thm:main2} we can actually extend Corollary \ref{cor:main} to vertical strip $\LLT$ polynomials (see Proposition \ref{prop:vertical}).\par

     The main idea to prove these results is to use the fact that $\csf_q$ and $\LLT$ are completely determined by certain linear relations and their values at the complete graphs, as proved in \cite{AN}. Consequently, it is enough to prove that the right-hand sides of Equations \eqref{eq:thm1} and $\eqref{eq:thm2}$ also satisfy these relations.\par

  \section{Preliminaries}

    \subsection{Symmetric functions}
    \label{sec:sym}
    Let $\Lambda$ be the algebra of symmetric functions, and $\Lambda_q:=\Lambda[q]$. As usual, we let $[j]_q=(q^j-1)/(q-1)$ and $j!_q=\prod_{i=1}^j[i]_q$. We denote by $e_\lambda$, $h_\lambda$, $s_\lambda$, and $p_\lambda$ the elementary, complete, Schur, and power sum symmetric functions. Also, we let $P_{\lambda}(q)\in \Lambda_q$ be the Hall-Littlewood polynomial (see \cite[Chapter 3]{Macdonald}). For our purposes, the relevant identity is
   \begin{equation}
       \label{eq:Ps}
     P_n(q)=\sum_{r=0}^{n-1} (-q)^rs_{n-r,1^r}.
   \end{equation}
   (see \cite[page 214]{Macdonald}). Then we define $\rho_n=q^{n-1}P_n(q^{-1})$, or alternatively $\rho_n=(-1)^{n-1}\omega(P_n)$, and $\rho_{\lambda}:=\prod \rho_{\lambda_i}$. Notice that $\rho_{\lambda}$ is, in general, not the same as $q^{|\lambda|-\ell(\lambda)}P_{\lambda}(q^{-1})$. We have that $\rho_\lambda$ interpolates between the power sum, $\rho_\lambda(1)=p_\lambda$, and the elementary symmetric functions, $\rho_\lambda(0)=(-1)^{n-\ell(\lambda)}e_\lambda$. \par

   Before stating a proposition which relates $\rho$ with the complete homogeneous symmetric functions $h_n$, we make a few definitions. A \emph{domino} is a connected horizontal strip, and a \emph{domino tabloid} of shape $\lambda$ and type $\mu$ is a filling of the Young diagram of $\lambda$ with dominoes of length $\mu_i$, where dominoes of the same length are indistinguishable. The $q$-weight of a domino tabloid is the product $\prod [l_i]_q$ where $l_i$ is the length of the leftmost domino in the $i$-th row of the Young diagram of $\lambda$. We denote by  $w_{\lambda\mu}$ the sum of the $q$-weights of all domino tabloids of shape $\lambda$ and type $\mu$.

     \begin{Prop}
     \label{prop:rhoeq}
       We have the following equalities.
       \begin{enumerate}
       \item \[
     [n]_qh_n=\sum_{j=1}^n h_{n-j}\rho_j,
     \]
     \item \[
     n!_qh_n=\det\left(\begin{array}{cccccc}
    \rho_{1}& \rho_{2} &\rho_{3}&\cdots&\rho_{n-1} &\rho_{n}\\
    -1& \rho_{1} &\rho_{2}&\cdots&\rho_{n-2} &\rho_{n-1}\\
    0 &-[2]_q& \rho_{1} &\cdots&\rho_{n-3}&\rho_{n-2}\\
    0 & 0 & -[3]_q &\cdots&\rho_{n-4}&\rho_{n-3}\\
    \vdots&\vdots &\vdots& \ddots & \vdots&\vdots\\
    0& 0 &0& \cdots & -[n-1]_q &\rho_{1}
        \end{array}
    \right)
    \]
        \item \[
    \rho_{n}=\det\left(\begin{array}{cccccc}
    h_{1}& {[2]}_qh_{2} &{[3]}_qh_{3}&\cdots&{[n-1]}_qh_{n-1} &{[n]}_qh_{n}\\
    1& h_{1} & h_{2} &\cdots& h_{n-2} &h_{n-1}\\
    0& 1 & h_{1} &\cdots& h_{n-3} &h_{n-2}\\
    0& 0 & 1 &\cdots& h_{n-4} &h_{n-3}\\
    \vdots&\vdots &\vdots& \ddots\ &\vdots &\vdots\\
    0& 0 & 0 &\cdots& 1 &h_{1}\\
    \end{array} \right)
    \]
    \item    \[
    \rho_n=\sum_{\alpha\models n}(-1)^{\ell(\alpha)-1}[\alpha_1]_qh_{\lambda(\alpha)}
    \]
   \item
   \[
   \rho_\lambda=\sum_{\mu\vdash n}(-1)^{\ell(\lambda)-\ell(\mu)}w_{\lambda\mu}h_{\mu}
   \]

    \end{enumerate}
    \end{Prop}
    \begin{proof}
    The first equality comes from Equation \eqref{eq:Ps}, the fact that $h_n=s_n$, and Pieri formula,
    \[
    s_{n-j}\cdot s_{j-r,1^r}=\sum_{a=0}^{\min(n-j,j-r-1)} s_{n-r-a,1+a,1^{r-1}}+ \sum_{a=0}^{min(n-j-1,j-r-1)} s_{n-r-a-1,1+a,1^{r}}.
    \]
    Items (2), (3) are a consequence of item (1) and the expansion of the determinant. Items (4) and (5) follow from item (1) just adapting a standard argument (\cite[page 109]{Macdonald}).
    \end{proof}

     \subsection{The modular law}
     \label{sec:csfllt}
      The chromatic symmetric function and the unicellular $\LLT$ polynomial satisfy a set of recurrences which characterizes them.  We denote by $\D$ the set of Hessenberg functions, which we identify with the set of indifference graphs. We recall that $\D$ has a multiplication given by $G_1\cdot G_2=G_1\cup G_2$, where $G_1\cup G_2$ is the ordered union of $G_1$ and $G_2$.
       \begin{Def}
      \label{def:modular}
      We say that a function $f\col \D\to \A$ satisfies the  \emph{modular law} if
      \begin{equation}
      \label{eq:rec}
      (1+q)f(\m_1)=qf(\m_0)+f(\m_2)
      \end{equation}
      whenever one of the following conditions hold
      \begin{enumerate}
          \item\label{item:1} There exists $i\in [n-1]$ such that $\m_1(i-1)<\m_1(i)<\m_1(i+1)$ and $\m_1(\m_1(i))=\m_1(\m_1(i)+1)$ or $\m_1(i)=n$. Moreover, $\m_0$ and $\m_2$ satisfy $\m_k(j):=\m_1(j)$ for every $j\neq i$ and $k=0,2$, while $\m_k(i)=\m_1(i)-1+k$.
          \item\label{item:2} There exists $i\in [n-1]$ such that $\m_1(i+1)=\m_1(i)+1$ and $\m_1^{-1}(i)=\emptyset$. Moreover, $\m_0$ and $\m_2$ satisfy $\m_k(j):=\m_1(j)$ for every $j\neq i,i+1$ and $k=0,2$, while $\m_0(i)=\m_0(i+1)=\m_1(i)$ and $\m_2(i)=\m_2(i+1)=\m_1(i+1)$.
      \end{enumerate}
      \end{Def}
      We have the following Theorems (see \cite[Theorem 1.2]{AN})

      \begin{Thm}
      \label{thm:csf}
      The chromatic quasisymmetric function is the unique multiplicative function $\csf_q\col \D\to \Lambda_q$ satisfying the modular law such that $\csf_q(K_n)=n!_qe_n$.
      \end{Thm}

      \begin{Thm}
      \label{thm:LLT}
      The unicellular $\LLT$ polynomial is the unique multiplicative function $\LLT\col \D\to \Lambda_q$ satisfying the modular law such that
      \[
      \LLT(K_n)=\sum_{i=1}^n \LLT(K_{n-i})((q-1)^{i-1}e_i)\prod_{j=n-i+1}^{n-1}[j]_q.
      \]
      \end{Thm}

      The recurrence for the complete graphs in Theorem \ref{thm:LLT} can be found in \cite[Proposition 5.18]{AlexPanova}.  The theorems above are equivalent via the plethystic relation (see \cite{CarlssonMellit})
      \[
      \frac{\LLT(G)[(q-1)X]}{(q-1)^n}=\csf_q(G).
      \]

    \subsection{Permutations and increasing forests}
    We begin by recalling the \emph{forget cycles} transformation $\sigma\to \sigma^c$, which simply consists in  writing $\sigma$ in cycle notation and forgetting the parenthesis. This means that each cycle must begin with its least element, and the cycles must be ordered according to their least element, for example
   \[
   \sigma=617892543, \quad \sigma=(162)(3759)(48), \quad\sigma^c=162375948.
   \]
   Denote by $\lambda(\sigma)$ the cycle partition of $\sigma$.\par

 For a permutation $\sigma \in S_n$ and Hessenberg function $\m\col[n]\to[n]$, we define the number of \emph{$\m$-inversions} of sigma as
 \[
 \inv_{\m}(\sigma):=|\{(i,j); i<j\leq \m(i), \sigma^{-1}(i)>\sigma^{-1}(j)\}|.
 \]
 We define the \emph{$\m$-weight} of a permutation $\sigma$ as $\wt_{\m}(\sigma):=\inv_{\m}( \sigma^c)$. We notice that $\wt_{\m}(\sigma)=\wt_{G_{\m}}(\sigma)$. Recall $\sigma\leq \m$ if $\sigma(i)\leq \m(i)$ for every $i\in[n]$, and $S_{n,\m}$ denotes the set $\{\sigma\in S_n; \sigma\leq \m\}$.\par
  \begin{Rem}
  \label{rem:taudet}
    We note that a $k$-cycle $\tau=(j_1,j_2,\ldots, j_k)$, with least element $j_1$, is in $S_{n,\m}$ if and only if $j_{l+1}\leq \m(j_l)$ for every $j\in [k-1]$. Moreover, a permutation $\sigma$ is in $S_{n,\m}$ if and only if all of its cycles are in $S_{n,\m}$ as well. Moreover, we have that the cycle $\tau\in S_{n,\m}$ on $\{j_1,\ldots, j_k\}$ is determined by the numbers
  \[
  \inv_{\m}(\tau,j_l)=|\{l'\in [k]; l'>l, j_{l'}<j_l\leq  \m(j_l')\}|,
  \]
  for every $l\in[k]$ (see \cite{GarsiaRemmel} or \cite[Section 9.5]{AlexPanova}).

   \end{Rem}


    Now we define the weight of an increasing spanning forest of an indifference graph $G$. Before doing that, we note that given a function $g\col[n] \to [n]\cup \{0\}$ such that $g(j)<j$, we have that the graph $F$ with vertex set $V(F)=[n]$ and edge set $E(F)=\{\{g(j),j\}; g(j)\neq0\}$ is an increasing forest. All increasing forests with vertex set $[n]$ can be obtained in this way. Given a increasing tree $T$, an edge $\{u,v\}$ with $u<v$ has a \emph{length} given by $|\{w\in V(T);u<w<v\}|$, and we define the \textit{weight} $\wt(T)$ of $T$ as the sum of the lengths of each edge. \par 
    Let $F=T_1\cup T_2\cup\ldots \cup T_k$ be an increasing spanning forest of a graph $G$. By convention, we assume that $\roott(T_1)<\roott(T_2)<\ldots<\roott(T_k)$. We say that a set $\{v,w\}$ of vertices is an \emph{inversion} of $F$ if $v<w$ and there exists $i<j$ such that $v\in T_j$ and $w\in T_i$. Moreover,  if $\{v,w\}$ is an edge of $G$, we say that $\{v,w\}$ is a $G$-\emph{inversion} of $F$ and define $\inv_G(F)$ as the number of $G$-inversions of $F$. The $G$-weight of $F$ is given by
\[
 \wt_G(F):=\inv_G(F)+\sum_{1\leq j\leq k}\wt(T_j).
 \]
 When $G$ is the graph induced by a Hessenberg function $\m$, we write $\wt_{\m}(F)$ for $\wt_G(F)$. For a forest $F$, we say that the partition of $F$ is the partition given by $|V(T_i)|$ and denote it by $\lambda(F)$. For a graph $G$, we denote by $\F(G)$ the set of increasing spanning forests of $G$.\par
  For example, considering the increasing forest in Figure \eqref{fig:permtree}, we have that $\wt(T_1)=3$, $\wt(T_2)=1$ and the set of inversions of $F$ is $\{\{3,4\},\{3,6\},\{3,8\},\{5,6\},\{5,8\},\{7,8\}\}$. See Figure \eqref{fig:permtree1}.

\begin{figure}[htb]
           \begin{tikzpicture}
              \begin{scope}
        \node (1) at (1,0) [label=below:$1$, shape=circle, fill=black, inner sep=2pt] {};
        \node (2) at (2,0) [label=below:$2$, shape=circle, fill=black,inner sep=2pt] {};
        \node (3) at (7,0) [label=below:$3$, shape=circle, fill=black,inner sep=2pt] {};
        \node (4) at (3,0) [label=below:$4$, shape=circle, fill=black,inner sep=2pt] {};
        \node (5) at (8,0) [label=below:$5$, shape=circle, fill=black,inner sep=2pt] {};
        \node (6) at (4,0) [label=below:$6$, shape=circle, fill=black,inner sep=2pt] {};
        \node (7) at (9,0) [label=below:$7$, shape=circle, fill=black,inner sep=2pt] {};
        \node (8) at (5,0) [label=below:$8$, shape=circle, fill=black, inner sep=2pt] {};
        \node (9) at (10,0) [label=below:$9$, shape=circle, fill=black, inner sep=2pt] {};
        \draw (1) to [bend left=45] (6);
        \draw (2) to  (4);
        \draw (1) to  (2);
        \draw[color=red] (3) to [bend left=45] (7);
        \draw[color=red] (3) to  (5);
        \draw[color=red] (7) to  (9);
        \draw (4) to [bend left=45] (8);
        \end{scope}
          \end{tikzpicture}
           \caption{Computing the weight and set of inversions of the increasing forest in Figure \eqref{fig:permtree}}
           \label{fig:permtree1}
       \end{figure}
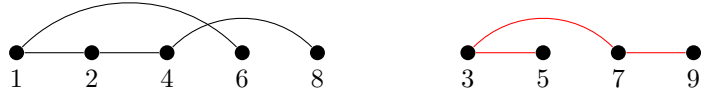

  The bijection between $S_{n,\m}$ and $\F(G_{\m})$ given in the introduction preserves partitions, but not weights. This is solved by the proposition below.

 \begin{Prop}
 Given a Hessenberg function $\m\col[n]\to[n]$, there exists a bijecion $S_{n,\m}\to \F(G_{\m})$ that preserves both weight and partition.  \par
 \end{Prop}
  \begin{proof}
    For each cycle $\tau=(j_1,\ldots, j_k)$ in $S_{n,\m}$, the idea is to construct a tree $T_{\tau}$ with vertex set $\{j_1,\ldots, j_k\}$ such that $\wt(T)=\inv_{\m}(\tau^c)$. We already know it is sufficient to define a function $g\col \{j_2,\ldots, j_k\}\to\{j_1,\ldots, j_k\}$ satisfying $g(j_l)<j_l$ for every $l\in\{2,\ldots, k\}$. In order to have $\wt(T_{\tau})=\inv_{\m}(\tau^c)$, we define $g$ as the unique function satisfying
  \[
  |\{l'\in [k], g(j_l)<j_{l'}<j_l\}|=|\{l'\in [k]; l'>l, j_{l'}<j_l\leq  \m(j_l')\}|
  \]
   for every $l\in\{2,\ldots, k\}$.\par
      To see that $T_{\tau}$ is indeed a subgraph of $G_{\m}$ it is sufficient to check that $\m(g(j_l))\geq j_l$ for every $l\in\{2,\ldots, k\}$. Assume for contradiction that there exists $l\in\{2,\ldots,k\}$ such that $\m(g(j_l))< j_l$, so in particular
  \[
  \{l'\in [k], g(j_l)<j_{l'}<j_l\}\supseteq\{l'\in [k]; j_{l'}<j_l\leq \m(j_{l'})\}\supseteq \{l'\in [k]; l'>l, j_{l'}<j_l\leq  \m(j_l')\}.
  \]
  Since $\{l'\in [k], g(j_l)<j_{l'}<j_l\}$ and $\{l'\in [k]; l'>l, j_{l'}<j_l\leq  \m(j_l')\}$ have the same cardinality, we must have equalities in the equation above.  However,  since $\tau\in S_{n,\m}$, we have that there must exist $l'<l$ such that $j_{l'}<j_l\leq \m(j_{l'})$, which means the last inclusion is proper, a contradiction.\par
   Let $\phi\col S_{n,\m}\to \F(G_{\m})$ be the function defined by $\phi(\sigma)=T_{\tau_1}\cup T_{\tau_2}\cup\ldots\cup T_{\tau_k}$ where $\tau_1,\ldots, \tau_k$ are the cycles of $\sigma$. The function $\phi$ is injective by Remark \ref{rem:taudet}, hence bijective because both sets have the same cardinality $\prod_{i=1}^n (\m(i)-i+1$), and preserves weight and partition by construction.
\end{proof}

 We denote the increasing forest associated to $\sigma$ via this construction by $F_{\sigma,\m}$. For example, if $\m=(3,3,5,5,5)$ and $\sigma=(13542)$, then $F_{\sigma,\m}$ is depicted in Figure \eqref{fig:Fsigma}.
  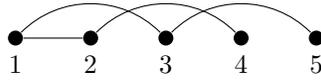
\begin{figure}[htb]
        \begin{tikzpicture}
        \begin{scope}
        \node (1) at (1,0) [label=below:$1$, shape=circle, fill=black, inner sep=2pt] {};
        \node (2) at (2,0) [label=below:$2$, shape=circle, fill=black,inner sep=2pt] {};
        \node (3) at (3,0) [label=below:$3$, shape=circle, fill=black,inner sep=2pt] {};
        \node (4) at (4,0) [label=below:$4$, shape=circle, fill=black,inner sep=2pt] {};
        \node (5) at (5,0) [label=below:$5$, shape=circle, fill=black,inner sep=2pt] {};
        \draw (3) to [bend left=45] (5);
        \draw (2) to [bend left=45] (4);
        \draw (1) to [bend left=45] (3);
        \draw (1) to  (2);
        \end{scope}
        \end{tikzpicture}
        \caption{The increasing tree induced by the cycle $(13542)\in S_{5,\m}$, where $\m=(3,3,5,5,5)$.}
        \label{fig:Fsigma}
     \end{figure}

\section{The symmetric function of increasing forests}
 In this section we define a symmetric function $X\col \D\to \Lambda_q$ that enumerates increasing spanning forests of indifference graphs. First, we choose generators $y_n\in \Lambda_q$ homogeneous of degree $n$ and let $y_{\lambda}=\prod y_{\lambda_i}$. These generators will either be $y_n=\rho_n$ or $y_n=(q-1)^{n-1}e_n$ in the next section.
 \begin{Def}
 \label{def:increasingfor}
   Let $\m$ be a Hessenberg function and $G_{\m}$ its associated graph. We define the following symmetric function
   \[
   X(\m)=X_y(\m):=\sum_{\sigma\leq \m} q^{\wt_{\m}(\sigma)}y_{\lambda(\sigma)}=\sum_{F\in \F(G_{\m})}q^{\wt_{\m}(F)}y_{\lambda(F)}=\sum_{\lambda \vdash n} c_\lambda(\m)y_\lambda
   \]
  \end{Def}
 \begin{Exa}
 As an example, if $\m=(2,3,4,4)$ we have $8$ increasing spanning forests (depicted in Figure \eqref{fig:forestspath})  of $G_{\m}$, all of which have weight $1$. We get
 \[
 X(\m)=y_{1,1,1,1}+3y_{2,1,1}+y_{2,2}+2y_{3,1}+y_4.
 \]
 More generally, if $\m=(2,3,\ldots, n-1,n,n)$ is the Hessenberg function associated to the \emph{path graph}, we have that the weight of any increasing spanning forest of $G_{\m}$ is $1$.
  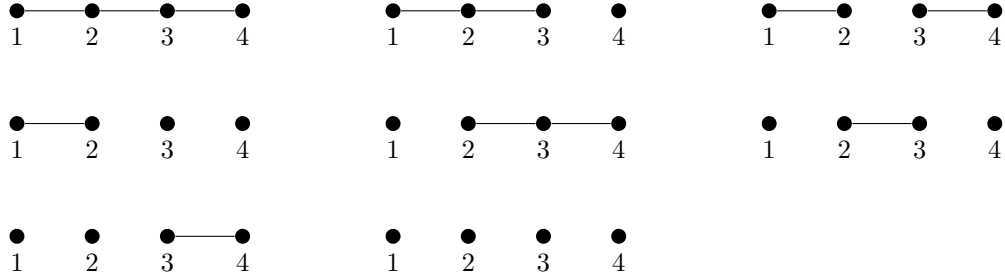
\begin{figure}[htb]
        \begin{tikzpicture}
        \begin{scope}
        \node (1) at (1,0) [label=below:$1$, shape=circle, fill=black, inner sep=2pt] {};
        \node (2) at (2,0) [label=below:$2$, shape=circle, fill=black,inner sep=2pt] {};
        \node (3) at (3,0) [label=below:$3$, shape=circle, fill=black,inner sep=2pt] {};
        \node (4) at (4,0) [label=below:$4$, shape=circle, fill=black,inner sep=2pt] {};
        \draw (1) to (2);
        \draw (2) to (3);
        \draw (3) to (4);
        \end{scope}
        \begin{scope}[shift={(5,0)}]
        \node (1) at (1,0) [label=below:$1$, shape=circle, fill=black, inner sep=2pt] {};
        \node (2) at (2,0) [label=below:$2$, shape=circle, fill=black,inner sep=2pt] {};
        \node (3) at (3,0) [label=below:$3$, shape=circle, fill=black,inner sep=2pt] {};
        \node (4) at (4,0) [label=below:$4$, shape=circle, fill=black,inner sep=2pt] {};
        \draw (1) to (2);
        \draw (2) to (3);
        \end{scope}
        \begin{scope}[shift={(10,0)}]
        \node (1) at (1,0) [label=below:$1$, shape=circle, fill=black, inner sep=2pt] {};
        \node (2) at (2,0) [label=below:$2$, shape=circle, fill=black,inner sep=2pt] {};
        \node (3) at (3,0) [label=below:$3$, shape=circle, fill=black,inner sep=2pt] {};
        \node (4) at (4,0) [label=below:$4$, shape=circle, fill=black,inner sep=2pt] {};
        \draw (1) to (2);
        \draw (3) to (4);
        \end{scope}
        \begin{scope}[shift={(0,-1.5)}]
        \node (1) at (1,0) [label=below:$1$, shape=circle, fill=black, inner sep=2pt] {};
        \node (2) at (2,0) [label=below:$2$, shape=circle, fill=black,inner sep=2pt] {};
        \node (3) at (3,0) [label=below:$3$, shape=circle, fill=black,inner sep=2pt] {};
        \node (4) at (4,0) [label=below:$4$, shape=circle, fill=black,inner sep=2pt] {};
        \draw (1) to (2);
        \end{scope}
        \begin{scope}[shift={(5,-1.5)}]
        \node (1) at (1,0) [label=below:$1$, shape=circle, fill=black, inner sep=2pt] {};
        \node (2) at (2,0) [label=below:$2$, shape=circle, fill=black,inner sep=2pt] {};
        \node (3) at (3,0) [label=below:$3$, shape=circle, fill=black,inner sep=2pt] {};
        \node (4) at (4,0) [label=below:$4$, shape=circle, fill=black,inner sep=2pt] {};
        \draw (2) to (3);
        \draw (3) to (4);
        \end{scope}
        \begin{scope}[shift={(10,-1.5)}]
        \node (1) at (1,0) [label=below:$1$, shape=circle, fill=black, inner sep=2pt] {};
        \node (2) at (2,0) [label=below:$2$, shape=circle, fill=black,inner sep=2pt] {};
        \node (3) at (3,0) [label=below:$3$, shape=circle, fill=black,inner sep=2pt] {};
        \node (4) at (4,0) [label=below:$4$, shape=circle, fill=black,inner sep=2pt] {};
        \draw (2) to (3);
        \end{scope}
        \begin{scope}[shift={(0,-3)}]
        \node (1) at (1,0) [label=below:$1$, shape=circle, fill=black, inner sep=2pt] {};
        \node (2) at (2,0) [label=below:$2$, shape=circle, fill=black,inner sep=2pt] {};
        \node (3) at (3,0) [label=below:$3$, shape=circle, fill=black,inner sep=2pt] {};
        \node (4) at (4,0) [label=below:$4$, shape=circle, fill=black,inner sep=2pt] {};
        \draw (3) to (4);
        \end{scope}
        \begin{scope}[shift={(5,-3)}]
        \node (1) at (1,0) [label=below:$1$, shape=circle, fill=black, inner sep=2pt] {};
        \node (2) at (2,0) [label=below:$2$, shape=circle, fill=black,inner sep=2pt] {};
        \node (3) at (3,0) [label=below:$3$, shape=circle, fill=black,inner sep=2pt] {};
        \node (4) at (4,0) [label=below:$4$, shape=circle, fill=black,inner sep=2pt] {};
        \end{scope}
        \end{tikzpicture}
        \caption{All the increasing spanning forests of $G_{\m}$ for $\m=(2,3,4,4)$.}
        \label{fig:forestspath}
     \end{figure}

 On the other hand, if $\m=(2,4,4,4)$, we have that $G_{\m}$ has twelve increasing spanning forests, four of which have partition $(3,1)$ and are depicted in Figure \eqref{fig:forests1333}. We get
 \[
 X(\m)=y_{1,1,1,1}+(q+3)y_{2,1,1}+y_{2,2}+(2q+2)y_{3,1}+(q+1)y_4
 \]
  \begin{figure}[htb]
        \begin{tikzpicture}
        \begin{scope}
        \node (1) at (1,0) [label=below:$1$, shape=circle, fill=black, inner sep=2pt] {};
        \node (2) at (2,0) [label=below:$2$, shape=circle, fill=black,inner sep=2pt] {};
        \node (3) at (3,0) [label=below:$3$, shape=circle, fill=black,inner sep=2pt] {};
        \node (4) at (4,0) [label=below:$4$, shape=circle, fill=black,inner sep=2pt] {};
        \draw (1) to (2);
        \draw (2) to (3);
        \draw (2) to [bend left=45] (4);
        \draw (3) to (4);
        \end{scope}
        \begin{scope}[shift={(5,0)}]
        \node (1) at (1,0) [label=below:$1$, shape=circle, fill=black, inner sep=2pt] {};
        \node (2) at (2,0) [label=below:$2$, shape=circle, fill=black,inner sep=2pt] {};
        \node (3) at (3,0) [label=below:$3$, shape=circle, fill=black,inner sep=2pt] {};
        \node (4) at (4,0) [label=below:$4$, shape=circle, fill=black,inner sep=2pt] {};
        \draw (1) to (2);
        \draw (2) to (3);
        \end{scope}
        \begin{scope}[shift={(10,0)}]
        \node (1) at (1,0) [label=below:$1$, shape=circle, fill=black, inner sep=2pt] {};
        \node (2) at (2,0) [label=below:$2$, shape=circle, fill=black,inner sep=2pt] {};
        \node (3) at (3,0) [label=below:$3$, shape=circle, fill=black,inner sep=2pt] {};
        \node (4) at (4,0) [label=below:$4$, shape=circle, fill=black,inner sep=2pt] {};
        \draw (1) to (2);
        \draw (2) to [bend left=45] (4);
        \end{scope}
        \begin{scope}[shift={(0,-1.5)}]
        \node (1) at (1,0) [label=below:$1$, shape=circle, fill=black, inner sep=2pt] {};
        \node (2) at (2,0) [label=below:$2$, shape=circle, fill=black,inner sep=2pt] {};
        \node (3) at (3,0) [label=below:$3$, shape=circle, fill=black,inner sep=2pt] {};
        \node (4) at (4,0) [label=below:$4$, shape=circle, fill=black,inner sep=2pt] {};
        \draw (2) to (3);
        \draw (3) to (4);
        \end{scope}
        \begin{scope}[shift={(5,-1.5)}]
        \node (1) at (1,0) [label=below:$1$, shape=circle, fill=black, inner sep=2pt] {};
        \node (2) at (2,0) [label=below:$2$, shape=circle, fill=black,inner sep=2pt] {};
        \node (3) at (3,0) [label=below:$3$, shape=circle, fill=black,inner sep=2pt] {};
        \node (4) at (4,0) [label=below:$4$, shape=circle, fill=black,inner sep=2pt] {};
        \draw (2) to (3);
        \draw (2) to [bend left=45] (4);
        \end{scope}
        \end{tikzpicture}
        \caption{The graph associated to $\m=(2,4,4,4)$ and the increasing spanning forests of $G_{\m}$ with partition $(3,1)$.}
        \label{fig:forests1333}
 \end{figure}
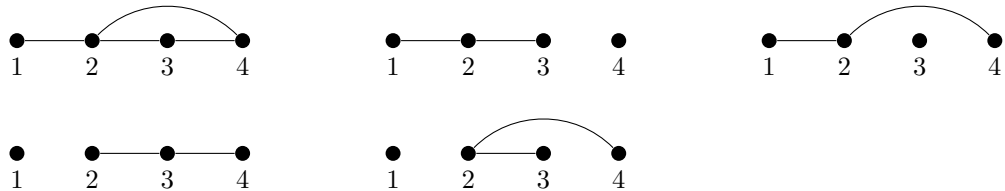

 When $\m=(3,4,4,5,5)$ there are eight increasing spanning forests of $G_{\m}$ with partition $(4,1)$, as depicted in Figure \eqref{fig:forests5}. We have
 \begin{align*}
 X(\m)=&y_{1,1,1,1,1}+(2q+4)y_{2, 1, 1, 1} + (2q+3)y_{2, 2, 1} + (q^2+4q+3)y_{3, 1, 1}\\
       &+ (2q+2)y_{3, 2} + (2q^2+4q+2)y_{4, 1} + (q^2+2q+1)y_5.
 \end{align*}
 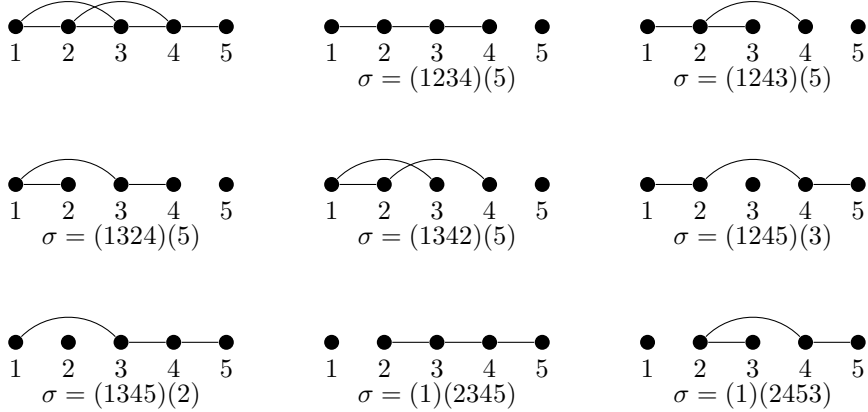
\begin{figure}[htb]
        \begin{tikzpicture}[scale=0.7]
        \begin{scope}
        \node (1) at (0,0) [label=below:$1$, shape=circle, fill=black, inner sep=2pt] {};
        \node (2) at (1,0) [label=below:$2$, shape=circle, fill=black,inner sep=2pt] {};
        \node (3) at (2,0) [label=below:$3$, shape=circle, fill=black,inner sep=2pt] {};
        \node (4) at (3,0) [label=below:$4$, shape=circle, fill=black,inner sep=2pt] {};
        \node (5) at (4,0) [label=below:$5$, shape=circle, fill=black,inner sep=2pt] {};
        \draw (1) to (2);
        \draw (1) to [bend left=45] (3);
        \draw (2) to (3);
        \draw (2) to [bend left=45] (4);
        \draw (3) to (4);
        \draw (4) to (5);
        \end{scope}
        \begin{scope}[shift={(6,0)}]
        \node (1) at (0,0) [label=below:$1$, shape=circle, fill=black, inner sep=2pt] {};
        \node (2) at (1,0) [label=below:$2$, shape=circle, fill=black,inner sep=2pt] {};
        \node (3) at (2,0) [label=below:$3$, shape=circle, fill=black,inner sep=2pt] {};
        \node (4) at (3,0) [label=below:$4$, shape=circle, fill=black,inner sep=2pt] {};
        \node (5) at (4,0) [label=below:$5$, shape=circle, fill=black,inner sep=2pt] {};
        \draw (1) to (2);
        \draw (2) to (3);
        \draw (3) to (4);
        \node at (2,-1) {$\sigma=(1234)(5)$};
        \end{scope}
        \begin{scope}[shift={(12,0)}]
        \node (1) at (0,0) [label=below:$1$, shape=circle, fill=black, inner sep=2pt] {};
        \node (2) at (1,0) [label=below:$2$, shape=circle, fill=black,inner sep=2pt] {};
        \node (3) at (2,0) [label=below:$3$, shape=circle, fill=black,inner sep=2pt] {};
        \node (4) at (3,0) [label=below:$4$, shape=circle, fill=black,inner sep=2pt] {};
        \node (5) at (4,0) [label=below:$5$, shape=circle, fill=black,inner sep=2pt] {};
        \draw (1) to (2);
        \draw (2) to (3);
        \draw (2) to [bend left=45] (4);
        \node at (2,-1) {$\sigma=(1243)(5)$};
        \end{scope}
        \begin{scope}[shift={(0,-3)}]
        \node (1) at (0,0) [label=below:$1$, shape=circle, fill=black, inner sep=2pt] {};
        \node (2) at (1,0) [label=below:$2$, shape=circle, fill=black,inner sep=2pt] {};
        \node (3) at (2,0) [label=below:$3$, shape=circle, fill=black,inner sep=2pt] {};
        \node (4) at (3,0) [label=below:$4$, shape=circle, fill=black,inner sep=2pt] {};
        \node (5) at (4,0) [label=below:$5$, shape=circle, fill=black,inner sep=2pt] {};
        \draw (1) to (2);
        \draw (1) to [bend left=45] (3);
        \draw (3) to (4);
        \node at (2,-1) {$\sigma=(1324)(5)$};
        \end{scope}
        \begin{scope}[shift={(6,-3)}]
        \node (1) at (0,0) [label=below:$1$, shape=circle, fill=black, inner sep=2pt] {};
        \node (2) at (1,0) [label=below:$2$, shape=circle, fill=black,inner sep=2pt] {};
        \node (3) at (2,0) [label=below:$3$, shape=circle, fill=black,inner sep=2pt] {};
        \node (4) at (3,0) [label=below:$4$, shape=circle, fill=black,inner sep=2pt] {};
        \node (5) at (4,0) [label=below:$5$, shape=circle, fill=black,inner sep=2pt] {};
        \draw (1) to (2);
        \draw (1) to [bend left=45] (3);
        \draw (2) to [bend left=45] (4);
        \node at (2,-1) {$\sigma=(1342)(5)$};
        \end{scope}
        \begin{scope}[shift={(12,-3)}]
        \node (1) at (0,0) [label=below:$1$, shape=circle, fill=black, inner sep=2pt] {};
        \node (2) at (1,0) [label=below:$2$, shape=circle, fill=black,inner sep=2pt] {};
        \node (3) at (2,0) [label=below:$3$, shape=circle, fill=black,inner sep=2pt] {};
        \node (4) at (3,0) [label=below:$4$, shape=circle, fill=black,inner sep=2pt] {};
        \node (5) at (4,0) [label=below:$5$, shape=circle, fill=black,inner sep=2pt] {};
        \draw (1) to (2);
        \draw (2) to [bend left=45] (4);
        \draw (4) to (5);
        \node at (2,-1) {$\sigma=(1245)(3)$};
        \end{scope}
        \begin{scope}[shift={(0,-6)}]
        \node (1) at (0,0) [label=below:$1$, shape=circle, fill=black, inner sep=2pt] {};
        \node (2) at (1,0) [label=below:$2$, shape=circle, fill=black,inner sep=2pt] {};
        \node (3) at (2,0) [label=below:$3$, shape=circle, fill=black,inner sep=2pt] {};
        \node (4) at (3,0) [label=below:$4$, shape=circle, fill=black,inner sep=2pt] {};
        \node (5) at (4,0) [label=below:$5$, shape=circle, fill=black,inner sep=2pt] {};
        \draw (1) to [bend left=45] (3);
        \draw (3) to (4);
        \draw (4) to (5);
        \node at (2,-1) {$\sigma=(1345)(2)$};
        \end{scope}
        \begin{scope}[shift={(6,-6)}]
        \node (1) at (0,0) [label=below:$1$, shape=circle, fill=black, inner sep=2pt] {};
        \node (2) at (1,0) [label=below:$2$, shape=circle, fill=black,inner sep=2pt] {};
        \node (3) at (2,0) [label=below:$3$, shape=circle, fill=black,inner sep=2pt] {};
        \node (4) at (3,0) [label=below:$4$, shape=circle, fill=black,inner sep=2pt] {};
        \node (5) at (4,0) [label=below:$5$, shape=circle, fill=black,inner sep=2pt] {};
        \draw (2) to (3);
        \draw (3) to (4);
        \draw (4) to (5);
        \node at (2,-1) {$\sigma=(1)(2345)$};
        \end{scope}
        \begin{scope}[shift={(12,-6)}]
        \node (1) at (0,0) [label=below:$1$, shape=circle, fill=black, inner sep=2pt] {};
        \node (2) at (1,0) [label=below:$2$, shape=circle, fill=black,inner sep=2pt] {};
        \node (3) at (2,0) [label=below:$3$, shape=circle, fill=black,inner sep=2pt] {};
        \node (4) at (3,0) [label=below:$4$, shape=circle, fill=black,inner sep=2pt] {};
        \node (5) at (4,0) [label=below:$5$, shape=circle, fill=black,inner sep=2pt] {};
        \draw (2) to (3);
        \draw (2) to [bend left=45] (4);
        \draw (4) to (5);
        \node at (2,-1) {$\sigma=(1)(2453)$};
        \end{scope}
      \end{tikzpicture}
      \caption{The graph associated to $\m=(3,4,4,5,5)$ and the eight increasing spanning forests of $G_{\m}$ with partition $(4,1)$, with their respective permutations.}
      \label{fig:forests5}
 \end{figure}

 \end{Exa}

 The rest of this section is devoted to prove that $X$ is multiplicative and satisfies the modular law, and to find a recurrence for the values $X(K_n)$ at complete graphs.

 \begin{Prop}
 \label{prop:multiplicative}
 The function $X$ is multiplicative.
 \end{Prop}
 \begin{proof}
 Let $G=G_1\cup G_2$ be the ordered disjoint union of $G_1$ and $G_2$. Clearly, every increasing spanning forest $F$ of $G$ is a union $F_1\cup F_2$ where $F_i$ is an increasing spanning forest of $G_i$ for $i=1,2$. Conversely, every such union $F_1\cup F_2$ is an increasing spanning forest of $G$. Moreover, we have that $\wt_G(F)=\wt_{G_1}(F_1)+\wt_{G_2}(F_2)$ and $y_{\lambda(F)}=y_{\lambda(F_1)}y_{\lambda(F_2)}$, hence
 \begin{align*}
 X(G)=&\sum_{F\in \F(G)}q^{\wt_{G}(F)}y_{\lambda(F)}\\
     =&\sum_{F_1\in \F(G_1)}\sum_{F_2\in \F(G_2)}q^{\wt_{G_1}(F_1)+\wt_{G_2}(F_2)}y_{\lambda(F_1)}y_{\lambda(F_2)}\\
     =&\bigg(\sum_{F_1\in \F(G_1)}q^{\wt_{G_1}(F_1)}y_{\lambda(F_1)}\bigg)\bigg(\sum_{F_2\in \F(G_2)}q^{\wt_{G_2}(F_2)}y_{\lambda(F_2)}\bigg)=X(G_1)X(G_2).
 \end{align*}
 This finishes the proof.
 \end{proof}

 \begin{Prop}
  \label{prop:forestmod}
 The function $X$ satisfies the modular law.
  \end{Prop}
  \begin{proof}
  Let $\m_1$, $\m_0$, $\m_2$, and $i$ as in Condition \ref{item:1}, and let $j:=\m_1(i)$. Let $\sigma\leq \m_1$ be a permutation and denote by $\sigma'=(j,j+1)\sigma(j,j+1)$ the conjugation of $\sigma$ with the transposition $(j,j+1)$. In particular, $\sigma'^c=(j,j+1)\sigma^c$ and $\sigma'$ have the same cycle partition as $\sigma$. Moreover, since $\m_1$, $\m_0$, $\m_2$, and $i$ satisfy Condition \ref{item:1}, we have that $\wt_{\m_2}(\sigma')=\wt_{\m_2}(\sigma)\pm1$ depending on the relative position of $j,j+1$ in $\sigma^c$.\par
   If $\sigma\leq \m_0$, then
  \[
  (1+q)q^{\wt_{\m_1}(\sigma)}=q\cdot q^{\wt_{\m_0}(\sigma)}+q^{\wt_{\m_2}(\tau)}
  \]
  where $\tau$ is chosen to be either $\sigma$ or $\sigma'$ depending on the relative position of $i,j,j+1$ in $\sigma^c$, as shown in Table \ref{tab:sigma}.
  \begin{table}[htb]
      \begin{tabular}{c|c|c|c|c}
                    & $\wt_{\m_1}(\sigma)$ & $\wt_{\m_0}(\sigma)$ & $\wt_{\m_2}(\sigma)$ &$\wt_{\m_2}(\sigma')$\\
            \hline
         $i$, $j$, $j+1$&0 &0&0&\\
         $i$, $j+1$, $j$ & 1 &1&1&\\
         $j$, $i$, $j+1$ & 1 &0&&2\\
         $j$, $j+1$, $i$&  1&0&2&\\
         $j+1$, $i$, $j$& 1 &1&&1\\
         $j+1$, $i$, $j$& 2 &1&3&
      \end{tabular}
      \caption{The contributions of $i,j,j+1$ to the weights of $\sigma$ and $\sigma'$ when $\sigma\leq \m_0$.}
      \label{tab:sigma}
     \end{table}

   If $\sigma\not\leq \m_0$, then $\sigma(i)=j$ which means that $i$ and $j$ are adjacent in $\sigma^c$. Then
   \[
  (1+q)q^{\wt_{\m_1}(\sigma)}= q^{\wt_{\m_2}(\sigma)}+q^{\wt_{\m_2}(\sigma')},
  \]
  as seen in Table \ref{tab:sigma1}.  This proves that $(1+q)X(\m_1)=qX(\m_0)+X(\m_2)$.
   \begin{table}[htb]
    \begin{tabular}{c|c|c|c}
      & $\wt_{\m_1}(\sigma)$ &  $\wt_{\m_2}(\sigma)$ &$\wt_{\m_2}(\sigma')$\\
            \hline
         $i$, $j$, $j+1$&0 &0&1\\
         $j+1$, $i$, $j$& 2 &3&2
      \end{tabular}
      \caption{The contributions of $i,j,j+1$ to the weights of $\sigma$ and $\sigma'$ when $\sigma\not\leq \m_0$.}
      \label{tab:sigma1}
  \end{table}

  Let $\m_1$, $\m_0$, $\m_2$, and $i$ as in Condition \ref{item:2}, and let $j:=\m_1(i+1)$. Let $\sigma\leq \m_1$ be a permutation and denote by $\sigma'=(i,i+1)\sigma(i,i+1)$. In particular $\sigma'^c=(i,i+1)\sigma^c$ and $\sigma'$ and $\sigma$ have the same cycle partition. If $\sigma\leq \m_0$, then
  \[
  (1+q)q^{\wt_{\m_1}(\sigma)}=q\cdot q^{\wt_{\m_0}(\sigma)}+q^{\wt_{\m_2}(\tau)}
  \]
  where $\tau$ is chosen to be either $\sigma$ or $\sigma'$ depending on the relative position of $i,i+1,j$ in $\sigma^c$, as shown in Table  \ref{tab:sigma2}.

  \begin{table}[htb]
      \begin{tabular}{c|c|c|c|c}
                   & $\wt_{\m_1}(\sigma)$ & $\wt_{\m_0}(\sigma)$ & $\wt_{\m_2}(\sigma)$ &$\wt_{\m_2}(\sigma')$\\
            \hline
         $i$, $i+1$, $j$&0 &0&0&\\
         $i$, $j$, $i+1$ & 1 &0&&2\\
         $i+1$, $i$, $j$ & 1 &1& 1&\\
         $i+1$, $j$, $i$&  1&1&&1\\
         $j$, $i$, $i+1$& 1 &0&2&\\
         $j$, $i+1$, $i$& 2 &1&3&
      \end{tabular}
      \caption{The contributions of $i,i+1,j$ to the weights of $\sigma$ and $\sigma'$ when $\sigma\leq \m_0$.}
      \label{tab:sigma2}
     \end{table}

   If $\sigma\not\leq \m_0$, then $\sigma(i+1)=j$ which means that $i+1$ and $j$ are adjacent in $\sigma^c$. Then
   \[
  (1+q)q^{\wt_{\m_1}(\sigma)}= q^{\wt_{\m_2}(\sigma)}+q^{\wt_{\m_2}(\sigma')},
  \]
  as seen in Table \ref{tab:sigma3}. This proves that $(1+q)X(\m_1)=qX(\m_0)+X(\m_2)$ and  finishes the proof.
 \begin{table}[htb]
    \begin{tabular}{c|c|c|c}
      & $\wt_{\m_1}(\sigma)$ &  $\wt_{\m_2}(\sigma)$ &$\wt_{\m_2}(\sigma')$\\
            \hline
         $i$, $i+1$, $j$&0 &0&1\\
         $i+1$, $j$, $i$& 1 &2&1
      \end{tabular}
      \caption{The contributions of $i,i+1,j$ to the weights of $\sigma$ and $\sigma'$ when $\sigma\not\leq \m_0$.}
      \label{tab:sigma3}
  \end{table}

  \end{proof}

  \begin{Prop}
  \label{prop:forestcomp}
  We have the following recursion for the complete graph.
  \[
  X(K_n)=\sum_{i=1}^{n} X(K_{n-i})y_{i}\prod_{j=n-i+1}^{n-1}[j]_q, \quad X(k_0)=1.
  \]
  \end{Prop}
  \begin{proof}
  We have that
  \[
  X(K_n)=\sum_{F\in \F(K_n)} q^{\wt_{K_n}(F)}y_{\lambda(F)}.
  \]
  We split this sum by the number of vertices of the tree $T_1$ with root $1$. If $T_1$ has size $i$ and other vertices $l_1,\ldots, l_{i-1}$, then the contribution of this tree to the weight of each forest containing $T_1$ is
  \[
  \wt(T_1)+\sum_{j=1}^{i-1} (l_j-j-1)(j-1).
  \]
  Hence
  \[
  \sum_{F\in F(K_n), T_1\subset F}q^{\wt_{K_n}(F)}y_{\lambda(F)}=q^{\wt(T_1)+\sum_{j=1}^{i-1} (l_j-j-1)(j-1)}y_iX(K_{n-i}).
  \]

  Varying $T_1$, we have that the sum $\sum_{1<l_1<l_2<\ldots<l_{i-1}\leq n}q^{\sum(l_j-j-1)(j-1)}$ is given by the $q$-binomial coefficient ${n-1\choose i-1}_q$ (see \cite{Polya}), while the sum of the weights of all increasing trees with vertex set $[i]$ is $(i-1)!_q$. Then
  \[
  \sum_{F\in F(K_n), |T_1|=i}q^{\wt_{K_n}(F)}y_{\lambda(F)} =(i-1)!_q{ n-1 \choose i-1}_qX(K_{n-i})y_i
  \]
  and the result follows.
   \end{proof}
   For a subset $I=\{i_1,\ldots, i_k\}\subset [n-1]$ we define $\lambda(I)$ as the conjugate of the partition associated to the composition $(i_1,i_2-i_1,\ldots, i_k-i_{k-1}, n-i_k)$.
\begin{Prop}
\label{prop:kn}
We have that
 \[
 X(K_n)=\sum_{I\subset [n-1]} y_{\lambda(I)} \prod_{j\in I}[j]_q.
 \]
 \end{Prop}
 \begin{proof}
 We will proceed by induction on $n$. The base case $n=0$ is trivial. For $I\subset [n-1]$ define $I_1$ and $I_2$ as $I_1=I\cap [n-i-1]$, $I_2=\{n-i+1,\ldots, n-1\}$ and $I=I_1\cup I_2$. Then $y_{\lambda(I)}=y_{\lambda_{I_1}}y_i$ and
 \begin{align*}
 \sum_{I\subset [n-1]} y_{\lambda(I)} \prod_{j\in I}[j]_q=&\sum_{i=1}^n\sum_{I_1\subset [n-i-1]}y_{\lambda(I_1)}y_i\prod_{j\in I_1}[j]_q\prod_{j=n-i+1}^{n-1}[j]_q\\
      =&\sum_{i=1}^n X(K_{n-i})y_i\prod_{j=n-i+1}^{n-i}[j]_q&\text{ by induction}\\
      =&X(K_n) & \text{ by Proposition \ref{prop:forestcomp}.}
 \end{align*}
  \end{proof}

  We can now relate the coefficients of $X(K_n)$ with the $q$-Stirling numbers of the first kind $s_q(n,k)$.
 \begin{Cor}
 \label{cor:stir}
 We have that
 \[
 \sum_{\lambda\vdash n, \ell(\lambda)=k}c_{\lambda}(K_n)=s_q(n,k).
 \]
 \end{Cor}
 \begin{proof}
This follows immediately by noticing that both sides are equal to
 \[
\underset{|I|=n-k}{\sum_{{I\subset [n-1]}}} \prod_{j\in I}[j]_q.
 \]
 \end{proof}

 Finally, we compute the sum of the coefficients  $c_{\lambda}(\m)$ for a given $\m\in \D$.
 \begin{Prop}
For $\m\in \D$, we have that
\[
\sum_{\sigma\leq \m}q^{\wt_{\m}(\sigma)}=\prod_{j=1}^{n}(1+[\m(i)-i]_q).
\]
  \end{Prop}
 \begin{proof}
 Both sides of the equation are multiplicative and satisfy the modular law. By \cite[Theorem 1.2]{AN}, it is enough to prove the Proposition for the complete graph. Let $k_n$ be the Hessenberg function associated to $K_n$, that is $k_n(i)=n$ for every $i\in [n]$. By Proposition \ref{prop:kn}, we have that
 \[
 \sum_{\sigma\leq k_n}q^{\wt_{\m}(\sigma)}=\sum_{I\subset [n-i]}\prod_{j\in I}[j]_q=\prod_{j=1}^{n-1}(1+[j]_q)=\prod_{j=1}^n(1+[k_n(j)-j]_q),
 \]
 which completes the proof.
  \end{proof}

\section{Relation with the chromatic quasisymmetric function and LLT polynomials}

   In this section we prove Theorems \ref{thm:main1}  and \ref{thm:main2}.
    \begin{proof}[Proof of Theorem \ref{thm:main1}]
    We begin by noticing that the right-hand side of Equation \ref{eq:thm1} is precisely $X_{\rho}(h)$.  Since both sides $\csf_q$ and $X$ are multiplicative and satisfy the modular law (by Theorem \ref{thm:csf}, Proposition \ref{prop:multiplicative} and Proposition \ref{prop:forestmod}), it is sufficient to prove the equality in the case of complete graphs. By Proposition \ref{prop:kn}, we have that
    \[
    X_\rho(K_n)=\det\left(\begin{array}{cccccc}
    \rho_{1}& \rho_{2} &\rho_{3}&\cdots&\rho_{n-1} &\rho_{n}\\
    -1& \rho_{1} &\rho_{2}&\cdots&\rho_{n-2} &\rho_{n-1}\\
    0 &-[2]_q& \rho_{1} &\cdots&\rho_{n-3}&\rho_{n-2}\\
    0 & 0 & -[3]_q &\cdots&\rho_{n-4}&\rho_{n-3}\\
    \vdots&\vdots &\vdots& \ddots & \vdots&\vdots\\
    0& 0 &0& \cdots & -[n-1]_q &\rho_{1}
        \end{array}
    \right)
    \]
    which by Proposition \ref{prop:rhoeq} means that $X_\rho(K_n)=n!_qh_n=\omega(\csf_q(K_n))$.
    \end{proof}

\begin{Rem}
In view of Theorem \ref{thm:main1} and Proposition \ref{prop:rhoeq} item (5), we can express the $e$-coefficients $\csf_{\lambda}$ of $\csf_q$ as a $Q(t)$-linear combination of $(c_\mu)_{\mu\vdash n}$. We have
\[
\csf_\lambda=\sum_{\mu\vdash n} (-1)^{\ell(\lambda)-\ell(\mu)}w_{\mu\lambda}c_{\mu},
\]
where $w_{\mu\lambda}$ is defined in Section \ref{sec:sym}. For example,
\begin{enumerate}
    \item if $\lambda=(n)$, then  $\csf_{n}=[n]_qc_\lambda$;
    \item if $\lambda=(a,b)$, with $b<a$, then $\csf_\lambda=[a]_q[b]_qc_\lambda-([a]_q+[b_q])c_{a+b}$;
    \item if $\lambda=(a,a)$, then $\csf_{\lambda}=[a]_q^2c_{\lambda}-[a]_qc_{n}$.
    \item if $\lambda$ has distinct parts $\lambda_1>\lambda_2>\ldots>\lambda_{\ell(\lambda})$ then
    \[
    w_{\mu\lambda}=\sum_{A}\prod_{B\in A}\sum_{j\in B}[\lambda_j]_q
    \]
    where the sum runs through all ordered set partitions $A=(B_1,\ldots, B_{\ell(\mu)})$ of $\{1,\ldots,\ell(\lambda)\}$ such that $\sum_{j\in B_k}\lambda_j=\mu_k$ for every $k=1,\ldots, \ell(\mu)$.
\end{enumerate}
\end{Rem}

We now prove Theorem \ref{thm:main2}.
  \begin{proof}[Proof of Theorem \ref{thm:main2}]
  If we take $y_n=(q-1)^{n-1}e_n$, then the right-hand side of Equation \eqref{eq:thm2} is precisely $X_y(G)$. By Theorem \ref{thm:LLT},  Proposition \ref{prop:multiplicative}, Proposition \ref{prop:forestmod}, and Proposition \ref{prop:forestcomp}, we have that both $\LLT$ and $X_y$ are multiplicative, satisfy the modular law, and have the same recurrence for complete graphs. This means that $\LLT=X_y$ and the result follows.
  \end{proof}

We now relate increasing spanning forests of an indifference graph $G$ to orientations of $G$, following \cite{Alexandersson_2020} and \cite{AS2020}.  We say that an oriented edge $\overrightarrow{uv}$ of $G$ is oriented to the right if $u<v$, and oriented to the left otherwise. Given an orientation $o$ of $G$, we remove all edges that are oriented to the right and consider the function $\lrv_o\col [n]\to [n]$ (called the \emph{lowest reaching vertex}), where $\lrv(v)$ is the minimum $u$ that is reachable from $v$. Consider the partition $\lambda(o)$ of $o$ as the partition induced by $|\lrv^{-1}(i)|$ for $i\in [n]$. Moreover, to each orientation $o$, we define the weight $\wt(o)$ as the number of edges oriented to the left. We will now give another proof of \cite[Corollary 2.10]{AS2020} for indifference graphs.

  \begin{Cor}
  \label{cor:peralex}
  If $G$ is an indifference graph, then $\LLT(G,q+1)=\sum_{o}q^{\wt(o)}e_{\lambda(o)}$, where the sum runs through all orientations of $G$.
  \end{Cor}
  \begin{proof}
  We construct a function $g\col \mathcal{O}(G_{\m})\to \F(G_{\m})$ in the following way. Let $o\in \mathcal{O}(G_{\m})$ be an orientation. Remove all edges that are oriented to the right, and all edges $\overrightarrow{vw}$ that are oriented to the left such that $\lrv_o(v)<\lrv_o(w)$ (this means that no path from $v$ to its lowest reaching vertex goes through the edge $\overrightarrow{vw}$). Now, for each vertex $v$, choose (if it exists) the edge with source $v$ and with greatest length. Define the forest $g(o)$ as the forest induced by the chosen edges. Clearly we have that $\lambda(o)=\lambda(g(0))$\par

  \begin{figure}[htb]
          \begin{tikzpicture}
          \begin{scope}[decoration={
    markings,
    mark=at position 0.5 with {\arrow{>}}}
    ]
        \node (1) at (1,0) [label=below:$1$, shape=circle, fill=black, inner sep=2pt] {};
        \node (2) at (2,0) [label=below:$2$, shape=circle, fill=black,inner sep=2pt] {};
        \node (3) at (3,0) [label=below:$3$, shape=circle, fill=black,inner sep=2pt] {};
        \node (4) at (4,0) [label=below:$4$, shape=circle, fill=black,inner sep=2pt] {};
        \node (5) at (5,0) [label=below:$5$, shape=circle, fill=black,inner sep=2pt] {};
        \node (6) at (6,0) [label=below:$6$, shape=circle, fill=black,inner sep=2pt] {};
        \node (7) at (7,0) [label=below:$7$, shape=circle, fill=black,inner sep=2pt] {};
        \draw[postaction=decorate] (5) to [bend right=45] (2);
        \draw[postaction=decorate] (3) to [bend left=45] (5);
        \draw[postaction=decorate] (5) to (4);
        \draw[postaction=decorate] (6) to [bend right=45] (4);
         \draw[postaction=decorate]  (1) to [bend left=45]   (3);
     \draw[postaction=decorate] (1) to [bend left=45] (4);
       \draw[postaction=decorate] (2) to  (3);
       \draw[postaction=decorate] (2) to [bend left=45] (4);
       \draw[postaction=decorate] (5) to [bend left=45] (7);
         \draw[postaction=decorate] (2) to  (1);
        \draw[postaction=decorate] (5) to [bend right=45] (1);
        \draw[postaction=decorate] (4) to (3);
        \draw[postaction=decorate] (6) to (5);
        \draw[postaction=decorate] (7) to (6);
        \end{scope}
        \begin{scope}[shift={(8,0)},decoration={
    markings,
    mark=at position 0.5 with {\arrow{>}}}
    ]
        \node (1) at (1,0) [label=below:$1$, shape=circle, fill=black, inner sep=2pt] {};
        \node (2) at (2,0) [label=below:$2$, shape=circle, fill=black,inner sep=2pt] {};
        \node (3) at (3,0) [label=below:$3$, shape=circle, fill=black,inner sep=2pt] {};
        \node (4) at (4,0) [label=below:$4$, shape=circle, fill=black,inner sep=2pt] {};
        \node (5) at (5,0) [label=below:$5$, shape=circle, fill=black,inner sep=2pt] {};
        \node (6) at (6,0) [label=below:$6$, shape=circle, fill=black,inner sep=2pt] {};
        \node (7) at (7,0) [label=below:$7$, shape=circle, fill=black,inner sep=2pt] {};
        \draw[postaction=decorate] (5) to [bend right=45] (2);

        \draw[postaction=decorate] (5) to (4);
        \draw[postaction=decorate] (6) to [bend right=45] (4);

         \draw[postaction=decorate] (2) to  (1);
        \draw[postaction=decorate] (5) to [bend right=45] (1);
        \draw[postaction=decorate] (4) to (3);
        \draw[postaction=decorate] (6) to (5);
        \draw[postaction=decorate] (7) to (6);

        \end{scope}
        \begin{scope}[shift={(0,-2.5)},decoration={
    markings,
    mark=at position 0.5 with {\arrow{>}}}
    ]
        \node (1) at (1,0) [label=below:$1$, shape=circle, fill=black, inner sep=2pt] {};
        \node (2) at (2,0) [label=below:$2$, shape=circle, fill=black,inner sep=2pt] {};
        \node (3) at (3,0) [label=below:$3$, shape=circle, fill=black,inner sep=2pt] {};
        \node (4) at (4,0) [label=below:$4$, shape=circle, fill=black,inner sep=2pt] {};
        \node (5) at (5,0) [label=below:$5$, shape=circle, fill=black,inner sep=2pt] {};
        \node (6) at (6,0) [label=below:$6$, shape=circle, fill=black,inner sep=2pt] {};
        \node (7) at (7,0) [label=below:$7$, shape=circle, fill=black,inner sep=2pt] {};
        \draw[postaction=decorate] (5) to [bend right=45] (2);
        \draw[postaction=decorate] (2) to  (1);
        \draw[postaction=decorate] (5) to [bend right=45] (1);
        \draw[postaction=decorate] (4) to (3);
        \draw[postaction=decorate] (6) to (5);
        \draw[postaction=decorate] (7) to (6);
        \end{scope}
        \begin{scope}[shift={(8,-2.5)},decoration={
    markings,
    mark=at position 0.5 with {\arrow{>}}}
    ]
        \node (1) at (1,0) [label=below:$1$, shape=circle, fill=black, inner sep=2pt] {};
        \node (2) at (2,0) [label=below:$2$, shape=circle, fill=black,inner sep=2pt] {};
        \node (3) at (3,0) [label=below:$3$, shape=circle, fill=black,inner sep=2pt] {};
        \node (4) at (4,0) [label=below:$4$, shape=circle, fill=black,inner sep=2pt] {};
        \node (5) at (5,0) [label=below:$5$, shape=circle, fill=black,inner sep=2pt] {};
        \node (6) at (6,0) [label=below:$6$, shape=circle, fill=black,inner sep=2pt] {};
        \node (7) at (7,0) [label=below:$7$, shape=circle, fill=black,inner sep=2pt] {};
        \draw[postaction=decorate] (2) to  (1);
        \draw[postaction=decorate] (5) to [bend right=45] (1);
        \draw[postaction=decorate] (4) to (3);
        \draw[postaction=decorate] (6) to (5);
        \draw[postaction=decorate] (7) to (6);
        \end{scope}
        \end{tikzpicture}
      \caption{From an orientation to an increasing tree. On the top right we remove all edges oriented to the right. On the bottom left, we remove all edges oriented to the left that do not reach the lowest reaching vertex. On the bottom right, we choose the edges with greatest length.}
      \label{fig:orentationtree}
  \end{figure}
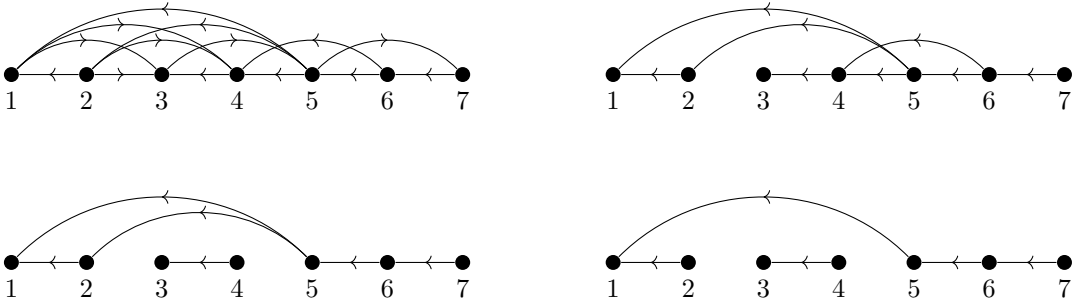

    We claim that for $F\in \F(G_{\m})$ the following equality holds (see Figure \eqref{fig:orientationforest})
  \[
  \sum_{o\in g^{-1}(F)} q^{\wt(o)}=q^{(n-\ell(\lambda(F)))}(q+1)^{\wt_{\m}(F)}.
  \]
  Indeed, for $o \in g^{-1}(F)$ each edge in $F$ must be oriented to the left (in Figure \eqref{fig:orientationforest} these are the edges in red), which contributes with $q^{n-\ell(\lambda(F))}$ to the weight of $o$. Each edge that connects two components of $F$ and that is not an inversion must be oriented to the right (in Figure \eqref{fig:orientationforest} these are the edges  $\{1,3\}$, $\{1,4\}$, $\{2,3\}$, and $\{2,4\}$). These edges do not contribute to the weight. Each edge that connects two components of $F$ and that is an inversion can be oriented either way (in Figure \eqref{fig:orientationforest}, these are the edges  $\{3,5\}$, $\{4,5\}$ and $\{4,6\}$). These edges contribute with $(q+1)^{\inv(F)}$ to the sum of the weights. The edges $\{u,v\}$, with $u<v$, that are not in $F$ but connect the same component $T$ in $F$, either have to be oriented right, if $\{u',v\}\in T$ for some $u<u'<v$ (in Figure \eqref{fig:orientationforest} this is the edge $\{5,7\}$), otherwise it can be oriented either way (in Figure \eqref{fig:orientationforest} this is the edge $\{2,5\}$ ). These edges contribute with $(q+1)^{\wt(T)}$ to the sum of the weights. \par
   By Theorem \ref{thm:main2} we have that
   \[
   \LLT(G_{\m};q+1)=\sum_{F\in F(G_{\m})}q^{n-\ell(\lambda(F))}(q+1)^{\wt_{\m}(F)}e_{\lambda(F)},
   \]
   from which the result follows.
    \end{proof}
  \begin{figure}[htb]
      \begin{tikzpicture}
         \begin{scope}
        \node (1) at (1,0) [label=below:$1$, shape=circle, fill=black, inner sep=2pt] {};
        \node (2) at (2,0) [label=below:$2$, shape=circle, fill=black,inner sep=2pt] {};
        \node (3) at (3,0) [label=below:$3$, shape=circle, fill=black,inner sep=2pt] {};
        \node (4) at (4,0) [label=below:$4$, shape=circle, fill=black,inner sep=2pt] {};
        \node (5) at (5,0) [label=below:$5$, shape=circle, fill=black,inner sep=2pt] {};
        \node (6) at (6,0) [label=below:$6$, shape=circle, fill=black,inner sep=2pt] {};
        \node (7) at (7,0) [label=below:$7$, shape=circle, fill=black,inner sep=2pt] {};
        \draw (1) to  (2);
        \draw (1) to [bend left=45] (3);
        \draw (1) to [bend left=45] (4);
        \draw (1) to [bend left=45] (5);
        \draw (2) to  (3);
        \draw (2) to [bend left=45] (4);
        \draw (2) to [bend left=45] (5);
        \draw (3) to (4);
        \draw (3) to [bend left=45] (5);
        \draw (4) to (5);
        \draw (4) to [bend left=45] (6);
        \draw (5) to (6);
        \draw (5) to [bend left=45] (7);
        \draw (6) to (7);
        \end{scope}
        \begin{scope}[shift={(8,0)}]
        \node (1) at (1,0) [label=below:$1$, shape=circle, fill=black, inner sep=2pt] {};
        \node (2) at (2,0) [label=below:$2$, shape=circle, fill=black,inner sep=2pt] {};
        \node (3) at (3,0) [label=below:$3$, shape=circle, fill=black,inner sep=2pt] {};
        \node (4) at (4,0) [label=below:$4$, shape=circle, fill=black,inner sep=2pt] {};
        \node (5) at (5,0) [label=below:$5$, shape=circle, fill=black,inner sep=2pt] {};
        \node (6) at (6,0) [label=below:$6$, shape=circle, fill=black,inner sep=2pt] {};
        \node (7) at (7,0) [label=below:$7$, shape=circle, fill=black,inner sep=2pt] {};
        \draw[color=red] (1) to  (2);
        \draw[color=red] (1) to [bend left=45] (5);
        \draw[color=red] (3) to (4);
        \draw[color=red] (5) to (6);
        \draw[color=red] (6) to (7);
        \end{scope}
        \begin{scope}[scale=1.5, shift={(1.333,-2)},decoration={
    markings,
    mark=at position 0.4 with {\arrow{<}}, mark=at position 0.6 with {\arrow{>}}}
    ]
        \node (1) at (1,0) [label=below:$1$, shape=circle, fill=black, inner sep=2pt] {};
        \node (2) at (2,0) [label=below:$2$, shape=circle, fill=black,inner sep=2pt] {};
        \node (3) at (3,0) [label=below:$3$, shape=circle, fill=black,inner sep=2pt] {};
        \node (4) at (4,0) [label=below:$4$, shape=circle, fill=black,inner sep=2pt] {};
        \node (5) at (5,0) [label=below:$5$, shape=circle, fill=black,inner sep=2pt] {};
        \node (6) at (6,0) [label=below:$6$, shape=circle, fill=black,inner sep=2pt] {};
        \node (7) at (7,0) [label=below:$7$, shape=circle, fill=black,inner sep=2pt] {};
        \draw[postaction=decorate] (2) to [bend left=45] (5);
        \draw[postaction=decorate] (3) to [bend left=45] (5);
        \draw[postaction=decorate] (4) to (5);
        \draw[postaction=decorate] (4) to [bend left=45] (6);

        \end{scope}
        \begin{scope}[scale=1.5, shift={(3,-3)},decoration={
    markings, mark=at position 0.55 with {\arrow{>}}},color=blue
    ]
     \draw[postaction=decorate]  (1) to [bend left=45]   (3);
     \draw[postaction=decorate] (1) to [bend left=45] (4);
       \draw[postaction=decorate] (2) to  (3);
       \draw[postaction=decorate] (2) to [bend left=45] (4);
       \draw[postaction=decorate] (5) to [bend left=45] (7);
    \end{scope}
    \begin{scope}[scale=1.5, shift={(3,-3)},decoration={
    markings, mark=at position 0.55 with {\arrow{<}}}, color=red
    ]
        \draw[postaction=decorate] (1) to  (2);
        \draw[postaction=decorate] (1) to [bend left=45] (5);
        \draw[postaction=decorate] (3) to (4);
        \draw[postaction=decorate] (5) to (6);
        \draw[postaction=decorate] (6) to (7);
    \end{scope}
    \end{tikzpicture}
       \caption{An indifference graph $G$, an increasing forest $F$ of $G$ and the orientations $o\in g^{-1}(G)$}
      \label{fig:orientationforest}
  \end{figure}
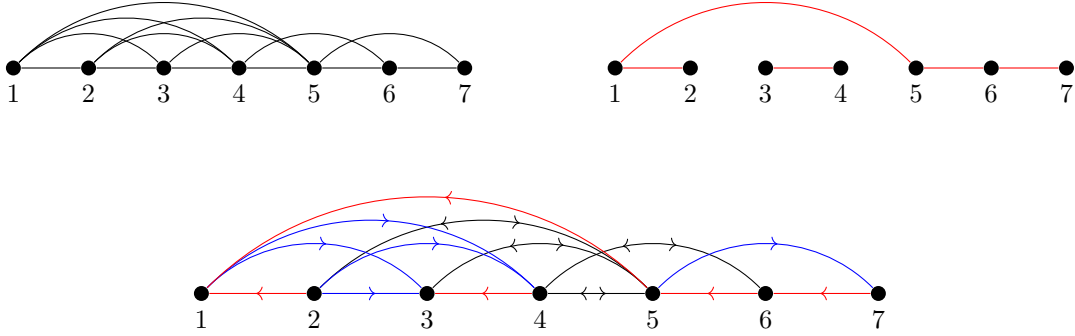

 We finish this section with a few comments about the vertical strip $\LLT$ polynomials. We refer the reader to \cite{LLT} for the original definition  and to \cite{AlexPanova} for the definition given here.\par
 Given a Hessenberg function $\m\col[n]\to[n]$ we say that a set $S\subset\{n-1\}$ is a decoration on $\m$ if $\m(i)>\max\{\m(i-1),i\}$ (recall that we set $\m(0)=0$) for every $i\in S$. We define $\LLT(\m,S)$ and $X_y(\m,S)$ by the following recursive formulas. If $S=\emptyset$, then $\LLT(\m,S)=\LLT(\m)$ and $X_y(\m,S)=X_y(\m)$, otherwise, for every $i\in S$ we have
  \begin{align*}
  \LLT(\m,S)&:=\frac{\LLT(\m,S')-\LLT(\m',S')}{q-1}\\
  X_y(\m,S)&:=X_y(\m,S')-X_y(\m',S'),
  \end{align*}
  where $S'=S\setminus\{i\}$ and
  \[
  \m'(j)=\begin{cases}
     \m(j)&\text{ if }j\neq i\\
     \m(i)-1&\text{ if }j=i
     \end{cases}
  \]
  In particular, if $y_n=(q-1)^{n-1}e_n$ then, by Theorem \ref{thm:main2}, we have that $\LLT(\m,S)=X_y(\m,S)/(q-1)^{|S|}$. As in Definition \ref{def:increasingfor} we write $X_y(\m,S)=\sum_{\lambda\vdash n}c_{\lambda}(\m,S)y_{\lambda}$.
  \begin{Prop}
  \label{prop:vertical}
  We have that $c_{\lambda}(\m,S,q+1)\in \mathbb{N}[q]$. In particular $\LLT(\m,S;q+1)$ is $e$-positive.
  \end{Prop}
  \begin{proof}
  For each subset $S'\subset S$ define $\m_{S'}$ as
  \[
  \m_{S'}(j)=\begin{cases}
  \m(j)&\text{ if }j\notin S'\\
  \m(j)-1&\text{ if }j\in S'.
  \end{cases}
  \]
  Then
  \[
  c_{\lambda}(\m,S)=\sum_{S'\subset S}(-1)^{|S'|}c_{\lambda}(\m_{S'})=\sum_{F\in F(G_\m)}\sum_{S'\subset S}(-1)^{|S'|}q^{\wt_{\m_{S'}}(F)}
  \]
  where we assume that $q^{\wt_{\m_{S'}}(F)}=0$ if $F\notin F(G_{\m_{S'}})$. However, for an increasing forest $F$, if there exists $i\in S$ such that the edge $(i,\m(i))$ is neither an edge of $F$ nor an inversion of $F$, then $q^{\wt_{\m_{S'}}}(F)=q^{\wt_{\m_{S'\cup\{i\}}}}(F)$ for every $S'\subset S\setminus\{i\}$. In particular, the sum $\sum_{S'\subset S}(-1)^{|S'|}q^{\wt_{\m_{S'}}(F)}$ is $0$. On the other hand, if the edge $(i,\m(i))$ is either an edge of $F$ or an inversion of $F$ for every $i\in S$, then
  \[
  \sum_{S'\subset S}(-1)^{|S'|}q^{\wt_{\m_{S'}}}(F)=(q-1)^k,
  \]
  where $k=|\{i\in S;(i,\m(i)) \text{ is a inversion of }F\}|$. This finishes the proof.
  \end{proof}

  \begin{Question}
    Does $c_{\lambda}(\m,S)\in \mathbb{N}[q]$?
    \end{Question}
    Computer experiments suggest that
    \[
    X(\m,S)=\sum_{F\in \F(G_{\m},S)}q^{\wt_{\m,S}(F)}e_{\lambda(F)},
    \]
    where $F(G_{\m},S)$ is the set of increasing spanning forests of $G_{\m}$ containing the edges $\{\{i,\m(i)\},i\in S\}$. However, we were not able to find the appropriate weight $\wt_{G_{\m,S}}(F)$.

\bibliographystyle{amsalpha}
\bibliography{bibhess}

\end{document}